\documentclass[10pt]{amsart}
\usepackage{amssymb}
\usepackage{latexsym}
\usepackage{euscript}
\usepackage{graphicx}
\usepackage[colorlinks,backref]{hyperref}

      \def\sL{{\mathfrak L}}
   \def\sN{{\mathfrak N}}

\def\bA{{\mathbb A}}      \def\dC{{\mathbb C}}

      \def\dR{{\mathbb R}}

\def\cD{{\mathcal D}}      
   \def\calH{{\mathcal H}}

\DeclareMathOperator{\Ext}{Ext}

\newcommand{\ti}{\tilde  }

\def\bm\chi{\mbox{\boldmath$\chi$}}
\def\half{{\frac{1}{2}}}

\def\RE{{\rm Re\,}}
\def\IM{{\rm Im\,}}

\def\dim{{\rm dim\,}}

\def\cmr{{\dC \setminus \dR}}

\newtheorem{theorem}{Theorem}[section]

\newtheorem{corollary}[theorem]{Corollary}

\theoremstyle{remark}
\newtheorem{remark}[theorem]{Remark}

\theoremstyle{definition}

\newtheorem{definition}[theorem]{Definition}

\numberwithin{equation}{section}

\begin{document}

\title[Stieltjes like functions and systems with Schr\"odinger operator]
{Stieltjes like functions and inverse problems for systems with Schr\"odinger operator}

\author[S.V.~Belyi]{Sergey Belyi}
\address{Department of Mathematics\\
Troy State University\\
Troy, AL 36082, USA} \email{sbelyi@trojan.troyst.edu}

\author[E.R.~Tsekanovski\u{\i}]{Eduard Tsekanovskii}
\address{Department of Mathematics\\
Niagara University, NY 14109 \\
USA} \email{tsekanov@niagara.edu}
\date{\today}



\begin{abstract}
A class of scalar Stieltjes like  functions is realized as
linear-fractio\-nal transformations of transfer functions of
conservative systems based  on a Schr\"odin\-ger operator $T_h$ in
$L_2[a,+\infty)$ with a non-selfadjoint boundary condition. In
particular it is shown that any Stieltjes  function of this class
can be realized in the unique way so that the main operator $\bA$ of
a system  is an accretive $(*)$-extension of a Schr\"odinger
operator $T_h$. We derive formulas that restore the system uniquely
and allow to find the exact value of a non-real parameter $h$ in the
definition of $T_h$ as well as a real parameter $\mu$ that appears
in the construction of the elements of the realizing system. An
elaborate investigation of these formulas shows the dynamics of the
restored parameters $h$ and $\mu$ in terms of the changing free term
$\gamma$ from the integral representation of the realizable
function. It turns out that the parametric equations for the
restored parameter $h$ represent different circles whose centers and
radii are determined by the realizable function. Similarly, the
behavior of the restored parameter $\mu$ are described by
hyperbolas.

\end{abstract}

\maketitle

\section{Introduction}\label{s-1}

Realizations of different classes of holomorphic operator-valued
functions in the open right half-plane, unit circle, and upper
half-plane, as well as inverse spectral problems,  play an important
role in the spectral analysis of non-self-adjoint operators,
interpolation problems, and system theory. The literature on
realization theory is too extensive to be discussed exhaustively in
this note. We refer, however, to \cite{AlGKS}, \cite{AlTs1},
\cite{AroDym1}, \cite{AroDym2}, \cite{BallSt1}, \cite{BT3},
\cite{BT4}, \cite{Hru}, \cite{St2}, \cite{Yurko} and the literature
therein. A class of Herglotz-Nevanlinna functions is a rich source
for many types of realization problems. An operator-valued function
$V(z)$ acting on a finite-dimensional Hilbert space $E$ belongs to
the class of operator-valued Herglotz-Nevanlinna functions if it is
holomorphic on $\cmr$, if it is symmetric with respect to the real
axis, i.e., $V(z)^*=V(\bar{z})$, $z\in \cmr$, and if it satisfies
the positivity condition
\[
 \IM V(z)\geq 0, \quad z\in \dC_+.
\]
It is well known (see e.g. \cite{GT}, \cite{GMKT}) that
operator-valued Herglotz-Nevanlinna functions admit the following
integral representation:
\begin{equation}
\label{nev0}
 V(z)=Q+Lz+\int_{\dR}
      \left( \frac{1}{t-z}-\frac{t}{1+t^2}\right)\, dG(t),
\quad
    z \in \cmr,
\end{equation}
where $Q=Q^*$, $L\geq 0$, and $G(t)$ is a nondecreasing
operator-valued function on $\dR$ with values in the class of
nonnegative operators in $E$ such that
\begin{equation}
\label{int0}
 \int_{\dR} \frac{\left( dG(t)x,x \right)_E}{1+t^2} <\infty,
 \quad x \in E.
\end{equation}
The realization of a selected class of Herglotz-Nevanlinna functions
is provided by a linear  conservative  system $\Theta$ of the form
\begin{equation}\label{Lyv}
\left\{   \begin{array}{l}
          (\bA-zI)x=KJ\varphi_-  \\
          \varphi_+=\varphi_- -2iK^* x
          \end{array}
\right.
\end{equation}
 or
\begin{equation}
\label{col0}
 \Theta =
\left(%
\begin{array}{ccc}
  \bA    & K & J \\
   \calH_+\subset\calH\subset\calH_- &  & E \\
\end{array}%
\right).
\end{equation}
In this system $\bA$, the \textit{main operator} of the system, is a
so-called ($*$)-extension, which is a bounded linear operator from
$\calH_+$ into $\calH_-$, where $\calH_+\subset\calH\subset\calH_-$
is a rigged Hilbert space. Moreover, $K$ is a bounded linear
operator from the finite-dimensional Hilbert space $E$ into
$\calH_-$, while $J=J^*=J^{-1}$ is acting on $E$, are such that
$\IM\bA=KJK^*$. Also, $\varphi_-\in E$ is an input vector,
$\varphi_+\in E$ is an output vector, and $x\in \calH_+$ is a vector
of the state space of the system $\Theta$. The system described by
\eqref{Lyv}-\eqref{col0} is called a rigged canonical system of the
Liv\v{s}ic type  \cite{Lv2} or the Brodski\u{\i}-Liv\v{s}ic rigged
operator colligation, cf., e.g. \cite{BT3}, \cite{BT4}, \cite{Br}.
The operator-valued function
\begin{equation}
\label{W1}
 W_\Theta(z)=I-2iK^*(\bA-zI)^{-1}KJ
\end{equation}
is a transfer function (or characteristic function) of the system
$\Theta$. It was shown in \cite{BT3} that an operator-valued
function $V(z)$ acting on a Hilbert space $E$ of the form
\eqref{nev0} can be represented and realized in the form
\begin{equation}
\label{real2}
 V(z)=i[W_\Theta(z)+I]^{-1}[W_\Theta(z)-I]
 =K^*(\bA_{R}-zI)^{-1}K,
\end{equation}
where $W_\Theta(z)$ is a transfer function  of some canonical
scattering ($J=I$) system $\Theta$, and where the ``\textit{real
part}'' $\bA_R=\half(\bA+\bA^*)$ of $\bA$ satisfies $\bA_R \supset
A$ if and only if the function $V(z)$ in \eqref{nev0} satisfies the
following two conditions:
\begin{equation}
\label{cond0} \left\{
\begin{array}{l}
 L = 0, \\
 Qx=\int_{\dR} \frac{t}{1+t^2}\, dG(t)x
   \quad \mbox{when} \quad
    \int_{\dR} \left(dG(t)x,x\right)_E <\infty.
\end{array}
\right.
\end{equation}

In the current paper we are going to focus on an important subclass
of  Herglotz-Nevanlinna functions, the so called Stieltjes like
functions that also includes Stieltjes functions.
In Section \ref{s-4} we specify a
subclass of realizable Stieltjes operator-functions and show that
any member of this subclass can be realized by a system of the form
\eqref{col0} whose main operator $\bA$   is accretive.

In  Section \ref{s-5} we introduce a  class of  Stieltjes like
scalar functions.  Then we rely on the general realization results
developed in Section \ref{s-4} (see also
\cite{DoTs}) to 
restore a system $\Theta$ of the form \eqref{col0} containing the
Schr\"odinger operator in $L_2[a,+\infty)$ with non-self-adjoint
boundary conditions
\begin{equation*}
 \left\{ \begin{array}{l}
 T_h y=-y^{\prime\prime}+q(x)y \\
 y^{\prime}(a)=hy(a) \\
 \end{array} \right., \quad \left(q(x)=\overline{q(x)},\,\IM h\ne0\right).
\end{equation*}
We show that if a non-decreasing function $\sigma(t)$ is the
spectral distribution function of positive self-adjoint boundary
value problem
\begin{equation*}
 \left\{ \begin{array}{l}
 A_\theta y=-y^{\prime\prime}+q(x)y \\
 y^{\prime}(a)=\theta y(a)
 \end{array} \right.
\end{equation*}
and satisfies conditions
\begin{equation*}
\int\limits_0^{\infty}d\sigma(t)=\infty,\quad
\int\limits_0^{\infty}\frac{d\sigma(t)}{1+t}<\infty,
\end{equation*}
then for every real $\gamma$ a Stieltjes like function
$$
V(z)=\gamma+\int\limits_0^{\infty}\frac{d\sigma(t)}{t-z}
$$
can be realized in the unique way as a $V_\Theta(z)$ function of a
rigged canonical system $\Theta$ containing some Schr\"odinger
operator $T_h$. In particular, it is shown that for every
$\gamma\ge0$ a Stieltjes function $V(z)$ with integral
representation above can be realized by a system $\Theta$ whose main
operator  $\bA$  is an accretive $(*)$-extension of a Schr\"odinger
operator $T_h$.

On top of the general realization results, Section \ref{s-5}
provides the reader with formulas that allow to find the exact value
of a non-real parameter $h$ in the definition of $T_h$   of the
realizing system  $\Theta$. Similar investigation is  presented in
Section \ref{s-6} to describe the real  parameter $\mu$ that appears
in the construction of the elements of the realizing system. A
detailed study of these formulas shows the dynamics of the restored
parameters $h$ and $\mu$ in terms of a changing free term $\gamma$
in the integral representation of $V(z)$ above. It will be shown and
graphically presented that the parametric equations for the restored
parameter $h$ represent different circles whose centers and radii
are completely determined by the function $V(z)$. Similarly, the
behavior of the restored parameter $\mu$ are described by
hyperbolas.

\section{Some preliminaries}\label{s-2}

For a pair of Hilbert spaces $\calH_1$, $\calH_2$ we denote by
$[\calH_1,\calH_2]$ the set of all bounded linear operators from
$\calH_1$ to $\calH_2$. Let $A$ be a closed, densely defined,
symmetric operator in a Hilbert space $\calH$ with inner product
$(f,g),f,g\in\calH$. Consider the rigged Hilbert space
$$\calH_+\subset\calH\subset\calH_- ,$$
where $\calH_+ =D(A^*)$ and
$$(f,g)_+ =(f,g)+(A^* f, A^*g),\;\;f,g \in D(A^*).$$
 Note that identifying the space
conjugate to $\calH_\pm$ with $\calH_\mp$, we get that if
$\bA\in[\calH_+,\calH_-]$ then $\bA^*\in[\calH_+,\calH_-].$
\begin{definition}
 An operator $\bA\in[\calH_+,\calH_-]$ is called a self-adjoint
bi-extension of a symmetric operator $A$ if $\bA=\bA^*$, $\bA
\supset A$, and the operator
$$\widehat A f = \bA f,\; f\in D(\widehat A)=\{f\in\calH_+:\bA f \in\calH\}$$
is self-adjoint in $\calH$.
\end{definition}
The operator $\widehat A$ in the above definition is called a
\textit{quasi-kernel} of a self-adjoint bi-extension $\bA$ (see
\cite{TSh1}) .
\begin{definition} An operator $\bA\in[\calH_+,\calH_-]$ is called a ($*$)-extension
(or correct bi-extension) of an operator $T$ (with non-empty set
$\rho(T)$ of regular points) if
$$\bA \supset T\supset A, \bA^* \supset T^* \supset A$$
and the operator $\bA_R=\frac {1}{2}(\bA+\bA^*)$ is a self-adjoint
bi-extension of an operator $A$.
\end{definition}

The existence, description, and analog of von Neumann's formulas for
self-adjoint bi-extensions and ($*$)-extensions were discussed in
\cite{TSh1} (see also \cite{arl2}, \cite{ArTs79}, \cite{BT3}). For
instance, if $\Phi$ is an isometric  operator from the defect
subspace $\sN_i$ of the symmetric operator $A$ onto the defect
subspace $\sN_{-i}$, then the formulas below establish a one-to one
correspondence between ($*$)-extensions of an operator $T$ and
$\Phi$
\begin{equation}
\label{111} \bA f=A^*f+iR(\Phi-I)x,\; \bA^* f=A^*f+iR(\Phi-I)y,
\end{equation}
where $x,y \in \sN_i$ are uniquely determined from the conditions
$$f-(\Phi+I)x\in D(T),\;f-(\Phi+I)y\in D(T^*)$$
and $R$ is the Riesz-Berezanskii operator of the triplet $\calH_+
\subset \calH \subset \calH_-$ that maps $\calH_+$ isometrically
onto $\calH_-$ (see \cite{TSh1}). If the symmetric operator $A$ has
deficiency indices $(n,n)$, then formulas \eqref{111} can be
rewritten in the following form
\begin{equation}
\label{112} \bA f = A^* f+ \sum\limits_{k = 1}^n {\Delta _k
(f)V_k},\quad \bA^* f = A^* f+ \sum\limits_{k = 1}^n {\delta _k
(f)V_k},
\end{equation}
where $\{V_j \}_1^n \in \calH_-$ is a basis in the subspace
$R(\Phi-I)\sN_i$, and $\{\Delta_k\}_1^n$, $\{\delta_k\}_1^n$, are
bounded linear functionals on $\calH_+$ with the properties
\begin{equation}
\label{113} \Delta_k(f)=0, \;\;\forall f \in D(T),\quad
\delta_k(f)=0,\; \;\forall f \in D(T^*).
\end{equation}

Let $\calH=L_2[a,+\infty)$ and $l(y)=-y^{\prime\prime}+q(x)y$ where
$q$ is a real locally summable function. Suppose that the symmetric
operator
\begin{equation}
\label{128}
 \left\{ \begin{array}{l}
 Ay=-y^{\prime\prime}+q(x)y \\
 y(a)=y^{\prime}(a)=0 \\
 \end{array} \right.
\end{equation}
has deficiency indices (1,1). Let $D^*$ be the set of functions
locally absolutely continuous together with their first derivatives
such that $l(y) \in L_2[a,+\infty)$. Consider $\calH_+=D(A^*)=D^*$
with the scalar product
$$(y,z)_+=\int_{a}^{\infty}\left(y(x)\overline{z(x)}+l(y)\overline{l(z)}
\right)dx,\;\; y,\;z \in D^*.$$ Let
$$\calH_+ \subset L_2[a,+\infty) \subset \calH_-$$
be the corresponding triplet of Hilbert spaces. Consider  operators
\begin{equation}
\label{131}
 \left\{ \begin{array}{l}
 T_hy=l(y)=-y^{\prime\prime}+q(x)y \\
 hy(a)-y^{\prime}(a)=0 \\
 \end{array} \right.
           ,\;\;  \left\{ \begin{array}{l}
 T^*_hy=l(y)=-y^{\prime\prime}+q(x)y \\
 \overline{h}y(a)-y^{\prime}(a)=0 \\
 \end{array} \right.,
\end{equation}
$$\left\{ \begin{array}{l}
 \widehat A y=l(y)=-y^{\prime\prime}+q(x)y \\
 \mu y(a)-y^{\prime}(a)=0 \\
 \end{array} \right.,\;\; \IM \mu =0. $$
It is well known \cite{AG93} that $\widehat A=\widehat{A^*}$. The
following theorem was proved in \cite{ArTs0}.
\begin{theorem}
\label{10} The set of all ($*$)-extensions of a non-self-adjoint
Schr\"odinger operator $T_h$ of the form \eqref{131} in
$L_2[a,+\infty)$ can be represented in the form
\begin{equation}
\label{137}
\begin{split}
&\bA y=-y^{\prime\prime}+q(x)y-\frac {1}{\mu-h}\,[y^{\prime}(a)-
hy(a)]\,[\mu \delta (x-a)+\delta^{\prime}(x-a)], \\
&\bA^* y=-y^{\prime\prime}+q(x)y-\frac {1}{\mu-\overline h}\,
[y^{\prime}(a)-\overline hy(a)]\,[\mu \delta
(x-a)+\delta^{\prime}(x-a)].
\end{split}
\end{equation}
In addition, the formulas \eqref{137} establish a one-to-one
correspondence between the set of all ($*$)-extensions of a
Schr\"odinger operator $T_h$ of the form \eqref{131} and all real
numbers $\mu \in [-\infty,+\infty]$.
\end{theorem}

\begin{definition} An operator $T$ with the domain $D(T)$ and $\rho(T)\ne\emptyset$ acting on a Hilbert space
$\calH$ is called \textit{accretive} if
$$\RE (Tf,f)\geq0,\; \;\forall f\in D(T).$$
\end{definition}

\begin{definition}  An accretive operator $T$ is called \cite{Ka}
\textit{$\alpha$-sectorial}  if there exists a value of
$\alpha\in(0,\pi/2)$ such that
\begin{equation*}\label{e8-29}
    \cot\alpha\,|\IM(Tf,f)|\le\RE(Tf,f),\qquad f\in\cD(T).
\end{equation*}
\end{definition}
An accretive operator is called \textit{extremal accretive} if it is
not $\alpha$-sectorial for any $\alpha\in(0,\pi/2)$.

Consider the symmetric operator $A$ of the form \eqref{128} with
defect indices (1,1), generated by the differential operation
$l(y)=-y^{\prime\prime}+q(x)y$. Let $\varphi_k(x,\lambda) (k=1,2)$
be the solutions of the following Cauchy problems:

\centerline{$\left\{ \begin{array}{l}
 l(\varphi_1)=\lambda \varphi_1 \\
 \varphi_1(a,\lambda)=0 \\
 \varphi'_1(a,\lambda)=1 \\
 \end{array} \right., $
$\left\{ \begin{array}{l}
 l(\varphi_2)=\lambda \varphi_2 \\
 \varphi_2(a,\lambda)=-1 \\
 \varphi'_2(a,\lambda)=0 \\
 \end{array} \right.. $}

It is well known \cite{AG93} that there exists a function
$m_\infty(\lambda)$ (called the  Weyl-Titchmarsh function) for which
$$\varphi(x,\lambda)=\varphi_2(x,\lambda)+m_\infty(\lambda)
\varphi_1(x,\lambda)$$ belongs to $L_2[a,+\infty)$.

Suppose that the symmetric operator $A$ of the form \eqref{128} with
deficiency indices (1,1) is nonnegative, i.e., $(Af,f) \geq 0$ for
all $f \in D(A))$. It was shown in \cite{Tse} that the Schr\"odinger
operator $T_h$ of the form \eqref{131} is accretive if and only if
\begin{equation}
\label{138} \RE h\geq-m_\infty(-0).
\end{equation}
For real $h$ such that $h\geq-m_\infty(-0)$ we get a description of
all nonnegative self-adjoint extensions of an operator $A$. For
$h=-m_\infty(-0)$ the corresponding operator
\begin{equation}
\label{139}
 \left\{ \begin{array}{l}
 A_K\,y=-y^{\prime\prime}+q(x)y \\
 y^{\prime}(a)+m_\infty(-0)y(a)=0 \\
 \end{array} \right.
\end{equation}
is the  Kre\u{\i}n-von Neumann extension of $A$ and for $h=+\infty$
the corresponding operator
\begin{equation}
\label{140}
 \left\{ \begin{array}{l}
 A_F\,y=-y^{\prime\prime}+q(x)y \\
 y(a)=0 \\
 \end{array} \right.
\end{equation}
is the Friedrichs extension of $A$ (see \cite{Tse}, \cite{ArTs0}).


\section{Rigged canonical systems with Schr\"odinger
operator}\label{s-3}

Let $\bA$ be ($*$) - extension of an operator $T$, i.e.,
$$\bA \supset T\supset A,\quad \bA^*\supset T^*\supset A$$
where $A$ is a symmetric operator with deficiency indices ($n,n$)
and $D(A)=D(T)\cap D(T^*)$. In what follows we will only consider
the case when the symmetric operator $A$ has dense domain, i.e.,
$\overline{\cD(A)}=\calH$.
\begin{definition}
A system of equations
\[
\left\{   \begin{array}{l}
          (\bA-zI)x=KJ\varphi_-  \\
          \varphi_+=\varphi_- -2iK^* x
          \end{array}
\right.,
\]
 or an
array
\begin{equation}\label{e6-3-2}
\Theta= \begin{pmatrix} \bA&K&\ J\cr \calH_+ \subset \calH \subset
\calH_-& &E\cr \end{pmatrix}
\end{equation}
 is called a \textit{rigged canonical system of the Livsic type}  or
 the
 \textit{Brodski\u{\i}-Livsic rigged operator colligation} if:

1) $E$ is a finite-dimensional Hilbert space with scalar product
$(\cdot,\cdot)_{E}$ and the operator $J$ in this space satisfies the
conditions $J=J^*=J^{-1}$,

2) $K\in [E,\calH_-]$,

3) $\IM\bA=KJK^*,$ where $K^*\in [\calH_+,E]$ is the adjoint of $K$.
\end{definition}
In the definition above   $\varphi_- \in E$ stands for an input
vector, $\varphi_+ \in E$ is an output vector, and $x$ is a state
space vector in $\calH$.
 An operator
$\bA$ is called a \textit{main operator} of the system $\Theta$, $J$
is a \textit{direction operator}, and $K$ is  a \textit{channel
operator}. An operator-valued function
\begin{equation}
\label{142}
   W_\Theta(\lambda)=I-2iK^*(\bA -\lambda I)^{-1}KJ
\end{equation}
defined on the set $\rho(T)$ of regular points of an operator $T$ is
called the \textit{transfer function} (\textit{characteristic
 function}) of the system $\Theta$, i.e.,
$\varphi_{+}= W_\Theta(\lambda)\varphi_{-}$. \noindent It is known
\cite{Tse},\cite{TSh1} that any $(*)$-extension $\bA$ of an operator
$T$ ($A^* \supset T \supset A)$, where $A$ is a symmetric operator
with deficiency indices $(n,n)$ $(n<\infty)$, $D(A)=D(T)\cap
D(T^*)$, can be included as a main operator of some rigged canonical
system with $\dim E<\infty$ and invertible channel operator $K$.

It was also established \cite{Tse}, \cite{TSh1} that
\begin{equation}
\label{143} V_\Theta(\lambda)=K^*(\RE \bA - \lambda I)^{-1}K
\end{equation}
is a Herglotz-Nevanlinna operator-valued function acting on a
Hilbert space $E$, satisfying the following relation for
$\lambda\in\rho(T),\; \IM\lambda\ne 0$
\begin{equation}
\label{144}
V_\Theta(\lambda)=i[W_\Theta(\lambda)-I][W_\Theta(\lambda)+I]^{-1}J.
\end{equation}
Alternatively,
\begin{equation}\label{e5-62}
\begin{aligned}
W_\Theta(\lambda)&=(I+iV_\Theta(\lambda)J)^{-1}(I-iV_\Theta(\lambda)J)\\
&=(I-iV_\Theta(\lambda)J)(I+iV_\Theta(\lambda)J)^{-1}.
\end{aligned}
\end{equation}
Let us recall  (see \cite{TSh1}, \cite{ArTs0}) that a  symmetric
operator with dense domain $\cD(A)$ is called \textit{prime} if
there is no reducing, nontrivial invariant subspace on which $A$
induces a self-adjoint operator. It was established in \cite{T87}
that a symmetric operator $A$ is prime if and only if
\begin{equation}\label{prime_op}
\mathop{c.l.s.}\limits_{\lambda\ne\overline\lambda}\sN_\lambda=\calH.
\end{equation}
 We call a
rigged canonical system of the form \eqref{e6-3-2} \textit{prime} if
$$\mathop{c.l.s.}\limits_{\lambda\ne\bar\lambda,\,\lambda\in\rho(T)}\sN_\lambda=\calH.$$
One easily verifies that if system $\Theta$ is prime, then a
symmetric operator $A$ of the system is  prime as well.


The following theorem \cite{ArTs0} establishes the connection
between two rigged canonical systems with equal transfer functions.
\begin{theorem}\label{12}
Let $\Theta_1= \begin{pmatrix} \bA_1&K_1&J\cr \calH_{+1} \subset
\calH_1 \subset \calH_{-1}& &E\cr \end{pmatrix} $
and \\
$\Theta_2= \begin{pmatrix} \bA_2&K_2&J\cr \calH_{+2} \subset \calH_2
\subset \calH_{-2}& &E\cr \end{pmatrix} $ be two prime rigged
canonical systems of the  Livsic type with
\begin{equation}
\label{165}
\begin{split}
&\bA_1 \supset T_1 \supset A_1,\quad
\bA_1^* \supset T_1^* \supset A_1, \\
&\bA_2 \supset T_2 \supset A_2,\quad \bA_2^* \supset T_2^* \supset
A_2,
\end{split}
\end{equation}
 and such that $A_1$ and $A_2$ have finite and equal defect
indices.

If
\begin{equation}
\label{166} W_{\Theta_1}(\lambda)=W_{\Theta_2}(\lambda),
\end{equation}
then there exists an isometric operator $U$ from $\calH_1$ onto
$\calH_2$ such that $U_+=U|_{\calH_{+1}}$ is an isometry\footnote{It
was shown in \cite{ArTs0} that the operator $U_+$ defined this way
is an isometry from $\calH_{+1}$ onto $\calH_{+2}$. It is also shown
there that the isometric operator
$U^*:\calH_{+2}\rightarrow\calH_{+1}$ uniquely defines operator
$U_-=(U^*)^*:\calH_{-1}\rightarrow\calH_{-2}$.} from $\calH_{+1}$
onto $\calH_{+2}$, $U_-^*=U_+^*$ is an isometry from $\calH_{-1}$
onto $\calH_{-2}$, and
\begin{equation}
\label{167} UT_1=T_2U,\quad \bA_2=U_-\bA_1U_+^{-1},\quad U_-K_1=K_2.
\end{equation}
\end{theorem}

\begin{corollary}\label{unique}
Let $\Theta_1$ and $\Theta_2$ be the two prime systems  from the
statement of theorem \ref{12}. Then the mapping $U$ described in the
conclusion of the theorem is unique.
\end{corollary}
\begin{proof}
First let us make an observation that if $\Theta= \begin{pmatrix}
\bA&K&J\cr \calH_{+} \subset \calH \subset \calH_{-}& &E\cr
\end{pmatrix} $
is a prime rigged canonical system such that $U_-\bA=\bA U_+$ and
$U_-K=K$, where $U$ is an isometry mapping described in theorem
\ref{12}, then $U=I$. Indeed, it is well known \cite{TSh1} that
\begin{equation}\label{spin}
    (\RE\bA-\lambda I)^{-1}KE=\sN_\lambda.
\end{equation}
We have
$$
\begin{aligned}
U(\RE\bA-\lambda I)^{-1}K&e=U_+(\RE\bA-\lambda
I)^{-1}Ke=(\RE\bA-\lambda I)^{-1}U_-Ke\\
&=(\RE\bA-\lambda I)^{-1}Ke,\qquad \forall e\in
E,\;\lambda\ne\bar\lambda.
\end{aligned}
$$
Combining the above equation with \eqref{prime_op} and \eqref{spin}
we obtain $U=I$.

Now let $\Theta_1$ and $\Theta_2$ be the two prime systems  from the
statement of theorem \ref{12}. Suppose there are two isometric
mappings $U_1$ and $U_2$ guaranteed by theorem \ref{12}. Then the
relations
$$
\bA_2=U_{-,1}\bA_1U_{+,1}^{-1},\quad U_{-,1}K_1=K_2,\quad
\bA_2=U_{-,2}\bA_1U_{+,2}^{-1},\quad U_{-,2}K_1=K_2,
$$
lead to
$$
\bA_1 U_{+,1}^{-1}U_{+,2}=U_{-,1}^{-1}U_{-,2}\bA_1,\quad
U_{-,1}^{-1}U_{-,2}K=K.
$$
Since $\Theta_1$ is prime then $U^{-1}_1U_2=I$ and hence $U_1=U_2$.
This proves the uniqueness of $U$.
\end{proof}

Now we shall construct a rigged canonical system based on a
non-self-adjoint Schr\"odinger operator. One can easily check that
the ($*$)-extension
\[
\bA y=-y^{\prime\prime}+q(x)y-\frac{1}{\mu-h}\,[y^{\prime}(a)
-hy(a)]\,[\mu \delta (x-a)+\delta^{\prime}(x-a)],\;\;\IM h>0
\]
of the non-self-adjoint Schr\"odinger operator $T_h$ of the form
\eqref{131} satisfies the condition
\begin{equation}
\label{145} \IM\bA=\frac{\bA - \bA^*}{2i}=(.,g)g,
\end{equation}
where
\begin{equation}
\label{146} g=\frac{(\IM h)^{\frac{1}{2}}}{|\mu - h|}\,
[\mu\delta(x-a)+\delta^{\prime}(x-a)]
\end{equation}
and $\delta(x-a), \delta^{\prime}(x)$ are the delta-function and its
derivative at the point a. Moreover,
\begin{equation}
\label{147} (y,g)=\frac{(\IM h)^{\frac{1}{2}}}{|\mu - h|}\ [\mu y(a)
-y^{\prime}(a)],
\end{equation}
where
$$y\in \calH_+, g\in \calH_-, \calH_+ \subset L_2(a,+\infty) \subset \calH_-$$ and
the triplet of Hilbert spaces is as discussed in theorem \ref{10}.
Let $E=\dC$, $K{c}=cg \;(c\in \dC)$. It is clear that
\begin{equation}
\label{148} K^* y=(y,g),\quad  y\in \calH_+
\end{equation}
and $\IM\bA=KK^*.$ Therefore, the array
\begin{equation}
\label{149} \Theta= \begin{pmatrix} \bA&K&1\cr \calH_+ \subset
L_2[a,+\infty) \subset \calH_-& &\dC\cr \end{pmatrix}
\end{equation}
is a rigged canonical system  with the main operator $\bA$ of the
form \eqref{137}, the direction operator $J=1$ and the channel
operator $K$ of the form \eqref{148}. Our next logical step is
finding the transfer function of \eqref{149}. It was shown in
\cite{ArTs0} that
\begin{equation}
\label{150} W_\Theta(\lambda)= \frac{\mu -h}{\mu - \overline h}\,\,
\frac{m_\infty(\lambda)+ \overline h}{m_\infty(\lambda)+h},
\end{equation}
and
\begin{equation}
\label{1501}
V_{\Theta}(\lambda)=\frac{\left(m_\infty(\lambda)+\mu\right)\IM h}
{\left(\mu-\RE h\right)m_\infty(\lambda)+\mu\RE h-|h|^2}.
\end{equation}


\section{Realization of Stieltjes functions}\label{s-4}

Let $E$ be a finite-dimensional Hilbert space. The scalar versions
of the following definition can be found in \cite{KK74}.

\begin{definition}\label{d8-4-1}
We will call an operator-valued Herglotz-Nevanlinna function
$V(z)\in [E,E]$ by a \textit{Stieltjes function}  if $V(z)$  admits
the following integral representation
\begin{equation}\label{e8-94}
V(z) =\gamma+\int\limits_0^\infty\frac {dG(t)}{t-z},
\end{equation}
where $\gamma\ge0$ and $G(t)$ is a non-decreasing on $[0,+\infty)$
operator-valued function such that
$$\int\limits^\infty_0\frac{(dG(t)e,e)_E}{1+t}<\infty,\quad \forall e\in E.$$
\end{definition}
Alternatively (see \cite{KK74}) an operator-valued function $V(z)$
is Stieltjes if it is holomorphic in $\Ext[0,+\infty)$ and
\begin{equation}\label{e4-0}
\frac{\IM[zV(z)]}{\IM z}\ge0.
\end{equation}


The theorem \ref{t8-20} below was stated in \cite{DTs}, \cite{DoTs}
and we present its proof for the  convenience of a reader.
\begin{theorem}\label{t8-20}
Let $\Theta$ be a prime  system of the form \eqref{e6-3-2}. Then an
operator-valued function $V_\Theta(z)$ defined by \eqref{143},
\eqref{144} is a Stieltjes function if and only if the main operator
$\mathbb A$ of the system $\Theta$ is accretive.
\end{theorem}
\begin{proof}
Let us assume first that $\mathbb A$ is an accretive operator, i.e.
$(\RE\bA x,x)\ge 0$, for all $x\in\calH_+$. Let $\{z_k\}$
($k=1,...,n$) be a sequence of non-real  complex numbers and $h_k$
be a sequence of vectors in $E$. Let us denote
\begin{equation}\label{e4-2'}
Kh_k=\delta_k,\quad x_k=(\RE\bA-z_kI)^{-1}\delta_k,\quad
x=\sum_{k=1}^n x_k.
\end{equation}
Since $(\RE\bA x,x)\ge0$, we have
\begin{equation}\label{e4-3'}
\sum_{k,l=1}^n (\RE\bA x_k,x_l)\ge0.
\end{equation}
By formal calculations one can have
$$
\RE\bA x_k=\delta_k +z_k(\RE\bA-z_k I)^{-1}\delta_k,
$$
and
$$
\begin{aligned}
\sum_{k,l=1}^n (\RE\bA x_k,x_l)&=\sum_{k,l=1}^n\big[(\delta_k,
(\RE\bA-z_l I)^{-1}\delta_l)\\
&+(z_k(\RE\bA-z_k I)^{-1}\delta_k,(\RE\bA-z_k I)^{-1}\delta_l)\big].
\end{aligned}
$$
Using obvious equalities
$$\big((\RE\bA-z_k I)^{-1}Kh_k,Kh_l\big)=\big(V_\theta(z_k)h_k,h_l\big)_E,$$
and
$$\big((\RE\bA-\bar z_l I)^{-1}(\RE\bA-z_k I)^{-1}Kh_k,Kh_l\big)=
\left(\frac{V_\theta(z_k)-V_\theta(\bar z_l)}{z_k-\bar
z_l}h_k,h_l\right)_E,
$$
we obtain
\begin{equation}\label{e4-4'}
\sum_{k,l=1}^n (\RE\bA
x_k,x_l)=\sum_{k,l=1}^n\left(\frac{z_kV_\theta(z_k)-\bar
z_lV_\theta(\bar z_l)}{z_k-\bar z_l}h_k,h_l\right)_E\ge0.
\end{equation}
The choice of $z_k$ was arbitrary, which means that $V_\Theta(z)$ is
a Stieltjes function (see \cite{AlTs1}).

Now we prove necessity. Since $\Theta$ is a prime system then $A$ is
a prime symmetric operator. Then the equivalence of (\ref{e4-4'})
and (\ref{e4-3'}) implies that $(\RE\bA x,x)\ge0$ for any $x$ from
c.l.s.$\left\{\mathfrak N_z\right\}$, $z\ne\bar z$. As we have
already mentioned above, a symmetric operator $A$ with the equal
deficiency indices is prime if and only if for all
$\lambda\ne\bar\lambda$
$$
\text{\rm c.l.s.}\left\{\mathfrak N_\lambda\right\}=\calH.
$$
Therefore we can conclude that $(\RE\bA x,x)\ge0$ for any
$x\in\calH_+$ and hence $\bA$ is an accretive operator.

\end{proof}

A system $\Theta $ of the form \eqref{e6-3-2} is called an
\textit{accretive system} if its main operator $\bA$ is accretive.

Now we define a certain class $S_0(R)$ of realizable Stieltjes
functions. At this point we need to note that since Stieltjes
functions form a subset of Herglotz-Nevanlinna functions then we can
utilize the conditions \eqref{cond0} to form a
 \textit{class $S(R)$} of
all \textit{realizable Stieltjes functions} (see also \cite{DoTs}).
Clearly, $S(R)$ is a subclass of $N(R)$ of all realizable
Herglotz-Nevanlinna functions described in details in \cite{BT3} and
\cite{BT4}.  To see the specifications of the class $S(R)$ we recall
that aside of integral representation \eqref{e8-94}, any Stieltjes
function admits a representation \eqref{nev0}. Applying condition
\eqref{cond0} we obtain
\begin{equation}\label{e8-100}
Q=\frac{1}{2}\left[V_\theta(-i)+V^\ast_\theta(-i)\right]=
\gamma+\int_{0}^{+\infty }\frac {t} {1+t^2}dG(t).
\end{equation}
Combining the second part of condition \eqref{cond0} and
(\ref{e8-100}) we conclude that
\begin{equation}\label{gamma0}
\gamma e=0,
\end{equation}
for all $e\in E$ such that
\begin{equation}\label{e8-99}
 \int_0^{\infty } (dG(t)e,e)_E < \infty.
\end{equation}
holds.  Consequently, \eqref{gamma0}-\eqref{e8-99} is precisely the
condition for $V(z)\in S(R)$.

 We are going to focus though on the subclass $S_0(R)$ of
$S(R)$ whose definition is the following.

\begin{definition}\label{d8-2} An operator-valued Stieltjes function $V(z)\in [E,E]$
is said to be a member of the \textbf{class
$S_0(R)$}\index{Class!$S_0(R)$} if in the representation
\eqref{e8-94} we have
\begin{equation}\label{e8-101}
 \int_0^{\infty } (dG(t)e,e)_E = \infty.
\end{equation}
for all non-zero $e\in E$.
\end{definition}
We note that a function $V(z)$ can belong to class $S_0(R)$ and have
an arbitrary constant $\gamma\ge0$ in the representation
\eqref{e8-94}.

 The following statement \cite{DoTs} is the direct realization theorem for the functions of the class $S_0(R)$.

\begin{theorem}\label{t8-17} Let $ \Theta $ be an accretive system
of the form \eqref{e6-3-2}. Then the operator-function $V_\Theta
(z)$ of the form \eqref{143}, \eqref{144} belongs to the class
$S_0(R)$.
\end{theorem}

\begin{proof}
To see that $V_\Theta(z)$ is a Stieltjes operator-function we merely
apply theorem \ref{t8-20} to system $\Theta$.

Now we will show that $V_\Theta(z)$ belongs to $S_0(R)$. It was
shown in \cite{BT3} and \cite{BT4} that $E_\infty=K^{-1}\mathfrak
L$, where $\mathfrak L=\calH \ominus \overline {\cD(A)}$ and
$$E_{\infty }=\left\{ e\in E:  \int_0^{+\infty }\left(
dG(t)e,e\right) _E< \infty \right\}.$$ But $\overline{\cD(A)}=\calH$
and consequently $\sL=\{0\}$. Next, $E_\infty=\{0\}$,
$$\int_{0}^{\infty } (dG(t)e,e)_E = \infty,$$
for all non-zero $e\in E$,  and therefore $V_\theta(z)\in S_0(R)$.

\end{proof}

The inverse realization theorem  can be stated and proved (see
\cite{DoTs}) for the classes $S_0(R)$ as follows.

\begin{theorem}\label{t8-18} Let a operator-valued function $V(z)$
belong to the class $S_0(R)$.  Then $V(z)$ admits a realization by
an accretive prime system $\Theta$ of the form \eqref{e6-3-2} with
$\cD(T)\ne\cD(T^\ast)$ and $J=I$.
\end{theorem}
\begin{proof} We have already noted that the class of Stieltjes function lies inside
the wider class of all Herglotz-Nevanlinna functions. Thus all we
actually have to show is  that $S_0(R)\subset N_0(R)$, where
$N_0(R)$ is subclass of realizable Herglotz-Nevanlinna functions
described in \cite{BT4}, and that the realizing system constructed
in \cite{BT4} appears to be an accretive system. The former is
rather obvious and follows directly from the definition of the class
$S_0(R)$. To see that the realizing system is accretive we need to
apply theorem \ref{t8-20} to $V_\theta(z)=V(z)$, where $V_\Theta(z)$
is related to the model system $\Theta$ that was constructed in
\cite{BT4}. As it was also shown in \cite{BT3} and \cite{BT4}, the
symmetric operator $A$ of the model system $\Theta$ is prime and
hence \eqref{prime_op} takes place. We are going to show that in
this case   the system $\Theta$ is also prime, i.e.,
\begin{equation}\label{prime_system}
\mathop{c.l.s.}\limits_{\lambda\ne\bar\lambda,\,\lambda\in\rho(T)}\sN_\lambda=\calH.
\end{equation}
Consider the operator $U_{\lambda_0\lambda}=(\ti A-\lambda_0I)(\ti
A-\lambda I)^{-1}$, where $\ti A$ is an arbitrary self-adjoint
extension of $A$. By a simple check one confirms  that
$U_{\lambda_0\lambda}\sN_{\lambda_0}=\sN_\lambda$. To prove
\eqref{prime_system} we assume that there is a function $f\in\calH$
such that
$$f\perp\mathop{c.l.s.}\limits_{\lambda\ne\bar\lambda,\,\lambda\in\rho(T)}\sN_\lambda.$$
Then $(f,U_{\lambda_0\lambda}g)=0$ for all $g\in\sN_{\lambda_0}$ and
all $\lambda\in\rho(T)$. Bur accretiveness of the system $\Theta$
implies that there are regular points of $T$ in the upper and lower
half-planes. This leads to a conclusion that the function
$\phi(\lambda)=(f,U_{\lambda_0\lambda}g)\equiv0$ for all
$\lambda\ne\bar\lambda$. Combining this with  \eqref{prime_op} we
conclude that $f=0$ and thus \eqref{prime_system} holds.
\end{proof}


\section{Restoring a non-self-adjoint Schr\"odinger operator $T_h$}\label{s-5}

In this section we are going to use the realization results for
Stieltjes functions developed in section \ref{s-4} to obtain the
solution of inverse spectral problem for  Schr\"odinger operator of
the form \eqref{131} in $L_2[a,+\infty)$ with non-self-adjoint
boundary conditions
\begin{equation}\label{e8-191}
 \left\{ \begin{array}{l}
 T_h y=-y^{\prime\prime}+q(x)y \\
 y^{\prime}(a)=hy(a) \\
 \end{array} \right., \quad \left(q(x)=\overline{q(x)},\,\IM h\ne0\right).
\end{equation}
In particular, we will show that if a non-decreasing function
$\sigma(t)$ is the spectral function of positive self-adjoint
boundary value problem
\begin{equation}\label{e8-192}
 \left\{ \begin{array}{l}
 A_\theta y=-y^{\prime\prime}+q(x)y \\
 y^{\prime}(a)=\theta y(a)
 \end{array} \right.
\end{equation}
and satisfies conditions
\begin{equation}\label{e8-192'}
\int\limits_0^{\infty}d\sigma(t)=\infty,\quad
\int\limits_0^{\infty}\frac{d\sigma(t)}{1+t}<\infty,
\end{equation}
then for every $\gamma\ge0$ a Stieltjes function
$$
V(z)=\gamma+\int\limits_0^{\infty}\frac{d\sigma(t)}{t-z}
$$
can be realized in the unique way as a $V_\Theta(z)$ function of an
accretive rigged canonical system $\Theta$ with some Schr\"odinger
operator $T_h$.

Let $\calH=L_2[a,+\infty)$ and $l(y)=-y^{\prime\prime}+q(x)y$ where
$q$ is a real locally summable function. We consider a symmetric
operator
\begin{equation}\label{B-sym}
 \left\{ \begin{array}{l}
 \ti B y=-y^{\prime\prime}+q(x)y \\
 y^{\prime}(a)=y(a)=0
 \end{array} \right.
\end{equation}
together with its positive self-adjoint extension of the form
\begin{equation}\label{e8-193}
 \left\{ \begin{array}{l}
 \ti B_\theta y=-y^{\prime\prime}+q(x)y \\
 y^{\prime}(a)=\theta y(a)
 \end{array} \right.
\end{equation}
defined in $\calH=L_2[a,+\infty)$. A non-decreasing function
$\sigma(\lambda)$ defined on $[0,+\infty)$ is called the
\textit{distribution function} (see \cite{Na68}) \textit{of an
operator pair} $\ti B_\theta$, $\ti B$, where
 $\ti B_\theta$ of the form \eqref{e8-193} is a self-adjoint extension of symmetric operator $\ti B$ of
 the form \eqref{B-sym}, and if the formulas
\begin{equation}\label{e-f-phi}
\begin{aligned}
\varphi(\lambda)&=Uf(x),\\
f(x)&=U^{-1}\varphi(\lambda),
\end{aligned}
\end{equation}
establish one-to-one isometric correspondence $U$ between
$L_2^\sigma[0,+\infty)$ and $L_2[a,+\infty)$. Moreover, this
correspondence is such that the operator $\ti B_\theta$ is unitarily
equivalent to the operator
\begin{equation}\label{Lambda-sigma}
\Lambda_\sigma \varphi(\lambda)=\lambda\varphi(\lambda),\quad
(\varphi(\lambda)\in L_2^\sigma[0,+\infty))
\end{equation}
in $L_2^\sigma[0,+\infty)$ while symmetric operator $\ti B$ in
\eqref{B-sym} is unitarily equivalent to the symmetric operator
\begin{equation}\label{Lambda-0}
\Lambda_\sigma^0 \varphi(\lambda)=\lambda\varphi(\lambda),\quad
\left(\varphi(\lambda)\in L_2^\sigma[0,+\infty),\;
\int_0^{+\infty}\varphi(\lambda)\,d\sigma(\lambda)=0\right).
\end{equation}

\begin{definition}\label{d8-like}
A scalar Herglotz-Nevanlinna function $V(z)$ is called
\textit{Stieltjes like function}  if it has an integral
representation \eqref{e8-94} with an arbitrary (not necessarily
non-negative) constant $\gamma$.
\end{definition}

We are going to introduce a new class of realizable scalar Stieltjes
like functions whose structure is similar to that of $S_0(R)$ of
section \ref{s-4}.
\begin{definition}\label{d8-6} A  Stieltjes
like function $V(z)$ is said to be a member of the \textit{class
$SL_0(R)$}  if it admits an integral representation
\begin{equation}\label{e8-195}
V(z)=\gamma+\int_0^{\infty } \frac{d\sigma(t)}{t-z},\qquad
\big(\gamma\in(-\infty,+\infty)\big),
\end{equation}
where non-decreasing function $\sigma(t)$ satisfies the following
conditions
\begin{equation}\label{e8-196}
 \int_0^{\infty } d\sigma(t) = \infty,\quad \int_0^{\infty }
 \frac{d\sigma(t)}{1+t}<\infty.
\end{equation}
\end{definition}
Consider the following subclasses of $SL_0(R)$.
\begin{definition}\label{d8-7} A  function $V(z)\in SL_0(R)$
belongs to the \textit{class $SL_0^K(R)$} if
\begin{equation}\label{e8-197}
\int_0^{\infty } \frac{d\sigma(t)}{t}=\infty.
\end{equation}
\end{definition}

\begin{definition}\label{d8-8} A  function $V(z)\in SL_0(R)$
belongs to the \textit{class $SL_{01}^K(R)$} if
\begin{equation}\label{e8-197'}
\int_0^{\infty } \frac{d\sigma(t)}{t}<\infty.
\end{equation}
\end{definition}

The following theorem describes the realization of the class
$SL_0(R)$.
\begin{theorem}\label{t8-33}
Let $V(z)\in SL_0(R)$ and the function $\sigma(t)$ be the
distribution function
 of an operator pair
 $\ti B_\theta$ of the form \eqref{B-sym} and $\ti B$ of the form \eqref{e8-193}. Then there exist unique
Schr\"odinger operator $T_h$ ($\IM h>0$) of the form \eqref{e8-191},
operator $\bA$ given by \eqref{137}, operator $K$ as in \eqref{148},
and the rigged canonical system of the Livsic type
\begin{equation}\label{e8-199}
\Theta= \begin{pmatrix} \bA&K&1\cr \calH_+ \subset L_2[a,+\infty)
\subset \calH_-& &\dC\cr \end{pmatrix},
\end{equation}
of the form \eqref{149} so that $V(z)$ is realized by $\Theta$.
\end{theorem}
\begin{proof}
Since $\sigma(t)$ is the distribution function of the positive
self-adjoint operator, then (see \cite{Na68}) we can completely
restore the operator $\ti B_\theta$ of the form \eqref{e8-193} as
well as a symmetric operator $\ti B$ of the form \eqref{B-sym}. It
follows from the definition of the distribution function above that
there is operator $U$ defined in \eqref{e-f-phi}  establishing
one-to-one isometric correspondence between $L_2^\sigma[0,+\infty)$
and $L_2[a,+\infty)$ while providing for the unitary equivalence
between the operator $\ti B_\theta$ and operator of multiplication
by independent variable $\Lambda_\sigma$ of the form
\eqref{Lambda-sigma}. Taking this into account, we realize (see
\cite{BT3}) a Herglotz-Nevanlinna function $V(z)$ with a rigged
canonical system
$$\Theta_\Lambda=\begin{pmatrix} \mathbf{\Lambda}&K^\sigma&1\cr \calH_+^\sigma \subset L_2^\sigma[0,+\infty)\subset
\calH_-^\sigma& &\dC\cr \end{pmatrix}.
$$
Following the steps for construction of the model system described
in \cite{BT3}, we note that
$$
\mathbf{\Lambda}=\RE\mathbf{\Lambda}+iK^\sigma (K^{\sigma })^*
$$
is a correct ($*$)-extension of an operator $T^\sigma$ such that
$\mathbf{\Lambda}\supset T^\sigma\supset\Lambda_\sigma^0$ where
$\Lambda_\sigma^0$ is defined in \eqref{Lambda-0}. The real part
$\RE\mathbf{\Lambda}$ is a self-adjoint bi-extension of
$\Lambda_\sigma^0$ that has a quasi-kernel $\Lambda_\sigma$ of the
form \eqref{Lambda-sigma}. The operator $K^\sigma$ in the above
system is defined by
$$
K^\sigma c=c\cdot 1,\quad c\in\dC,\;1\in\calH_-^\sigma.
$$
Here we need to clarify why number $1$ belongs to $\calH_-^\sigma$.
To confirm this we need to show that $(x,1)$ defines a continuous
linear functional for every $x\in\calH_+^\sigma$. It was shown in
\cite{BT3}, \cite{BT4} that
$$
\calH_+^\sigma=\cD(\Lambda_\sigma^0)\dotplus\left\{
\frac{c_1}{1+t^2}\right\}\dotplus\left\{
\frac{c_2t}{1+t^2}\right\},\quad c_1,c_2\in\dC.
$$
Consequently, every vector $x\in\calH_+^\sigma$ has three components
$x=x_1+x_2+x_3$ according to the decomposition above. Obviously,
$(x_1,1)$ and $(x_2,1)$ yield convergent integrals while $(x_3,1)$
boils down to
$$
\int_0^{\infty }\frac{t}{1+t^2}\, d\sigma(t).
$$
To see the convergence of the above integral we notice that
$$
\frac{t}{1+t^2}=\frac{t-1}{(1+t^2)(t+1)}+\frac{1}{1+t}\le\frac{1}{1+t^2}+\frac{1}{1+t}.
$$
The integrals taken of the last two expressions on the right side
converge due to \eqref{int0} and \eqref{e8-196}, and hence so does
the integral of the left side. Thus, $(x,1)$ defines a continuous
linear functional for every $x\in\calH_+^\sigma$ and
$1\in\calH_-^\sigma$.

 The
state space of the system $\Theta_\Lambda$ is $\calH_+^\sigma
\subset L_2^\sigma[0,+\infty)\subset \calH_-^\sigma$, where
$\calH_+^\sigma=\cD\big((\Lambda_\sigma^0)^*\big)$. By the
realization theorem \cite{BT3} we have that
$V(z)=V_{\Theta_\Lambda}(z)$.

We can also show that the system $\Theta_\Lambda$ is a prime system.
In order to do so we need to show that
\begin{equation}\label{prime_Lambda}
\mathop{c.l.s.}\limits_{\lambda\ne\bar\lambda,\,\lambda\in\rho(T^\sigma)}\sN_\lambda=L_2^\sigma[0,+\infty),
\end{equation}
where $\sN_\lambda$ are defect subspaces of the symmetric operator
$\Lambda_\sigma^0$. It is known (see \cite{BT3}) that
$\Lambda_\sigma^0$ is a prime operator. Hence we can follow the
reasoning of the proof of theorem \ref{t8-18} and only confirm  that
operator $T^\sigma$ has regular points in the upper and lower
half-planes. To see this we first note that non-negative operator
$\Lambda_\sigma^0$ has no kernel spectrum \cite{AG93} on the left
real half-axis. Then we apply Theorem 1 of \cite{AG93} (see page 149
of vol. 2 of \cite{AG93}) that gives the complete description of the
spectrum of $T^\sigma$. This theorem implies that there are regular
points of $T^\sigma$ on the left real half-axis. Since
$\rho(T^\sigma)$ is an open set we confirm the presence of non-real
regular points of $T^\sigma$ in both half-planes. Thus
\eqref{prime_Lambda} holds and $\Theta_\Lambda$ is a prime system.

 Applying  theorem \ref{12} on unitary equivalence to the
isometry $U$ defined in \eqref{e-f-phi} we obtain a triplet of
isometric operators $U_+$, $U$, and $U_-$, where
$$
U_+=U\big|_{\calH_+^\sigma},\quad U_-^*=U_+^*.
$$
This triplet of isometric operators will map the rigged Hilbert
space $\Theta_\Lambda$ is $\calH_+^\sigma \subset
L_2^\sigma[0,+\infty)\subset \calH_-^\sigma$ into another rigged
Hilbert space  $\calH_+\subset L_2^\sigma[a,+\infty)\subset
\calH_-$. Moreover, $U_+$ is an isometry from
$\calH_+^\sigma=\cD(\Lambda_\sigma^{0*})$ onto $\calH_{+}=\cD(\ti
B^*)$, and $U_-^*=U_+^*$ is an isometry from $\calH_+^\sigma$ onto
$\calH_{-}$. This is true since the operator $U$ provides the
unitary equivalence between the symmetric operators $\ti B$ and
$\Lambda_\sigma^0$.

Now we construct a system
$$
\Theta= \begin{pmatrix} \bA&K&1\cr \calH_+ \subset L_2[a,+\infty)
\subset \calH_-& &\dC\cr \end{pmatrix}
$$
where  $K=U_-K^\sigma$ and $\bA=U_-\mathbf{\Lambda}U_+^{-1}$ is a
correct ($*$)-extension of operator $T=UT^\sigma U^{-1}$ such that
$\bA\supset T\supset\ti B$. The real part $\RE\bA$ contains the
quasi-kernel $\ti B_\theta$. This construction of $\bA$ is unique
due to the theorem on the uniqueness of a ($*$)-extension for a
given quasi-kernel (see \cite{TSh1}). On the other hand, all
($*$)-extensions based on a pair $\ti B$, $\ti B_\theta$ must take
form \eqref{137} for some values of parameters $h$ and $\mu$.
Consequently, our function $V(z)$ is realized by the system $\Theta$
of the form \eqref{e8-199} and
$$V(z)=V_{\Theta_\Lambda}(z)=V_\Theta(z).$$
\end{proof}
\begin{remark}\label{r5-5}
Applying corollary \ref{unique} to the mapping $U$ defined by
\eqref{e-f-phi} we obtain that the operator $U$ in the above theorem
is unique. The uniqueness of the operator $U$ leads to an
interesting observation.
 Let  $u_k(x,\lambda)$,
$(k=1,2)$ be the solutions of the following Cauchy problems:

\centerline{$\left\{ \begin{array}{l}
 l(u_1)=\lambda u_1 \\
 u_1(a,\lambda)=0 \\
 u'_1(a,\lambda)=1 \\
 \end{array} \right., $\qquad
$\left\{ \begin{array}{l}
 l(u_2)=\lambda u_2 \\
 u_2(a,\lambda)=1 \\
 u'_2(a,\lambda)=0 \\
 \end{array} \right.. $}
\noindent Traditionally, (see \cite{Na68}) a non-decreasing function
$\sigma(\lambda)$ defined on $[0,+\infty)$ is called the
\textit{distribution function  of a self-adjoint operator} $\ti
B_\theta$ of the form \eqref{e8-193} if the formulas
\begin{equation}\label{e-f-phi-prime}
\begin{aligned}
\varphi(\lambda)&=Uf(x)=\int_a^{+\infty}f(x)u(x,\lambda)\,dx,\\
f(x)&=U^{-1}\varphi(\lambda)=\int_0^{+\infty}\varphi(\lambda)u(x,\lambda)\,d\sigma(\lambda),
\end{aligned}
\end{equation}
where $u(x,\lambda)=u_1(x,\lambda)+\theta u_2(x,\lambda)$, establish
one-to-one isometric correspondence $U$ between
$L_2^\sigma[0,+\infty)$ and $L_2[a,+\infty)$ such that the operator
$\ti B_\theta$ in \eqref{e8-193} is unitarily equivalent to the
operator $\Lambda_\sigma$ in \eqref{Lambda-sigma}. It is easily seen
that if the mapping $U$ in \eqref{e-f-phi-prime} is such that
symmetric operators $\ti B$ in \eqref{B-sym} and $\Lambda_\sigma^0$
in \eqref{Lambda-0} are unitarily equivalent w.r.t. $U$ as well,
then the mapping $U$ in theorem \ref{t8-33} is given by the formulas
\eqref{e-f-phi-prime}. Indeed, assuming that there is another
mapping $\ti U$ provided by   theorem \ref{12} on unitary
equivalence for the systems $\Theta_\Lambda$ and $\Theta$ we would
violate the uniqueness of the operator $U$, and thus $\ti U=U$.
\end{remark}

\begin{theorem}\label{t8-34}
Let $V(z)\in SL_0(R)$ satisfy the conditions of theorem \ref{t8-33}.
Then the operator $T_h$ in the conclusion of the theorem \ref{t8-33}
is accretive if and only if
\begin{equation}\label{e8-200}
    \gamma^2+\gamma\,\int_0^{\infty } \frac{d\sigma(t)}{t}+1\ge0.
\end{equation}
The operator $T_h$ is $\alpha$-sectorial for some
$\alpha\in(0,\pi/2)$ if and only if the inequality \eqref{e8-200} is
strict. In this case the exact value of angle $\alpha$ can be
calculated by the formula
\begin{equation}\label{e8-201}
 \tan\alpha=\frac{\int_0^{\infty } \frac{d\sigma(t)}{t}}{ \gamma^2+\gamma\,
\int_0^{\infty } \frac{d\sigma(t)}{t}+1}.
\end{equation}
\end{theorem}
\begin{proof}
It was shown in \cite{T87} that for the system $\Theta$ in
\eqref{e8-199} described in the previous theorem the operator $T_h$
is accretive if and only if the function
\begin{equation}\label{e8-202}
    \begin{aligned}
V_h(z)&=-i[W_\Theta^{-1}(-1)W_\Theta(z)+I]^{-1}[W_\Theta(-1)W_\Theta(z)-I]\\
&=-i\frac{1-[(m_\infty(z)+\bar h)/(m_\infty(z)+h)][(m_\infty(-1)+
h)/(m_\infty(-1)+\bar h)]}{1+[(m_\infty(z)+\bar
h)/(m_\infty(z)+h)][(m_\infty(-1)+ h)/(m_\infty(-1)+\bar h)]},
    \end{aligned}
\end{equation}
is holomorphic in Ext$[0,+\infty)$ and satisfies the following
inequality
\begin{equation}\label{e8-203}
    1+V_h(0)\,V_h(-\infty)\ge0.
\end{equation}
Here $W_\Theta(z)$ is the transfer function of \eqref{e8-199}. It is
also shown in \cite{T87} that the operator $T_h$ is
$\alpha$-sectorial for some $\alpha\in(0,\pi/2)$ if and only if the
inequality \eqref{e8-203} is strict while  the exact value of angle
$\alpha$ can be calculated by the formula
\begin{equation}\label{e8-204}
 \cot\alpha=\frac{1+V_\Theta(0)\,V_\Theta(-\infty)}{|V_\Theta(-\infty)-V_\Theta(0)|}.
\end{equation}

According to theorem  \ref{t8-33} and equation \eqref{e5-62}
$$
W_\Theta(z)=(I-iV(z)J)(I+iV(z)J)^{-1}.
$$
By direct calculations one obtains
\begin{equation}\label{e8-216}
W_\Theta(-1)=\frac{1-i\left[\gamma+\int_0^{\infty }
\frac{d\sigma(t)}{t+1}\right]}{1+i\left[\gamma+\int_0^{\infty }
\frac{d\sigma(t)}{t+1}\right]},\;
W_\Theta^{-1}(-1)=\frac{1+i\left[\gamma+\int_0^{\infty }
\frac{d\sigma(t)}{t+1}\right]}{1-i\left[\gamma+\int_0^{\infty }
\frac{d\sigma(t)}{t+1}\right]}.
\end{equation}
Using the following notations
$$
a=\int_0^{\infty } \frac{d\sigma(t)}{t+1} \textrm{\quad and \quad}
b=\int_0^{\infty } \frac{d\sigma(t)}{t},
$$
and performing straightforward  calculations we obtain
\begin{equation}\label{e8-217}
V_h(0)=\frac{a-b}{1+ab}\textrm{\quad and \quad}
V_h(-\infty)=\frac{a-\gamma}{1+a\gamma}.
\end{equation}
Substituting \eqref{e8-217} into \eqref{e8-204} and performing the
necessary steps we get
\begin{equation}\label{e8-218}
 \cot\alpha=\frac{1+b\gamma}{b-\gamma}=\frac{ \gamma^2+\gamma\,
\int_0^{\infty } \frac{d\sigma(t)}{t}+1}{\int_0^{\infty }
\frac{d\sigma(t)}{t}}.
\end{equation}
Taking into account that $b-\gamma>0$ we combine \eqref{e8-203},
\eqref{e8-204} with \eqref{e8-218} and this completes the proof of
the theorem.
\end{proof}

\begin{corollary}\label{c8-7}
Let $V(z)\in SL_0(R)$ satisfy the conditions of  theorem
\ref{t8-33}. Then the operator $T_h$ in the conclusion of theorem
\ref{t8-33}  is accretive if and only if
\begin{equation}\label{e8-219}
     1+V(0)\,V(-\infty)\ge0.
\end{equation}
The operator $T_h$ is $\alpha$-sectorial for some
$\alpha\in(0,\pi/2)$ if and only if the inequality \eqref{e8-219} is
strict. In this case the exact value of angle $\alpha$ can be
calculated by the formula
\begin{equation}\label{e8-220}
 \tan\alpha=\frac{V(-\infty)-V(0)}{1+V(0)\,V(-\infty)}.
\end{equation}
\end{corollary}
\begin{proof}
  Taking into account that
$$
V(0)=\gamma+\int_0^{\infty } \frac{d\sigma(t)}{t},
$$
$V(z)=V_\Theta(z)$, and $V_\Theta(-\infty)=\gamma$, we use
\eqref{e8-200} and \eqref{e8-201} to  obtain \eqref{e8-219} and
\eqref{e8-220}.
\end{proof}

\begin{theorem}\label{t8-35}
Let $V(z)\in S_0(R)$ and satisfy the conditions of theorem
\ref{t8-33}. Then the system $\Theta$ of the form \eqref{e8-199} is
accretive and its symmetric operator $A$ of the form \eqref{128} is
such that its Kre\u{\i}n-von Neumann extension $A_K$ of the form
\eqref{139} does not coincide with its Friedrichs extension $A_F$ of
the form \eqref{140}.
\end{theorem}
\begin{proof}
The proof of the fact that $\Theta$ is accretive directly follows
from the theorems \ref{t8-20} and \ref{t8-33}. The second part
follows from the theorem in \cite{Tse} that states that a positive
symmetric operator $A$ admits a non-self-adjoint accretive extension
$T$ if and only if $A_F$ and $A_K$ do not coincide.
\end{proof}

Below we will derive the formulas for calculation of the boundary
parameter $h$ in the restored Schr\"odinger operator $T_h$  of the
form \eqref{e8-191}. We consider two major cases.

\textbf{Case 1}. In the first case we assume that $\int_0^{\infty }
\frac{d\sigma(t)}{t}<\infty$. This means that our function $V(z)$
belongs to the class $SL_{01}^K(R)$. In what follows we denote
$$
b=\int_0^{\infty } \frac{d\sigma(t)}{t} \textrm{\quad and \quad}
m=m_\infty(-0).
$$
Suppose that $b\ge2$. Then the quadratic inequality \eqref{e8-200}
implies that for all $\gamma$ such that
\begin{equation}\label{gamma-int}
\gamma\in\left(-\infty,\frac{-b-\sqrt{b^2-4}}{2}
\right]\cup\left[\frac{-b+\sqrt{b^2-4}}{2},+\infty \right)
\end{equation}
the restored operator $T_h$ is accretive. Clearly, this operator is
extremal accretive if
$$
\gamma=\frac{-b\pm\sqrt{b^2-4}}{2}.
$$
In particular if $b=2$ then $\gamma=-1$ and the function
$$
V(z)=-1+\int_0^{\infty } \frac{d\sigma(t)}{t-z}
$$
is realized using an extremal accretive $T_h$.

Now suppose that $0<b<2$. For every $\gamma\in(-\infty,+\infty)$ the
restored operator $T_h$ will be accretive and $\alpha$-sectorial for
some $\alpha\in(0,\pi/2)$. Consider a function $V(z)$ defined by
\eqref{e8-195}. Conducting realizations of $V(z)$ by operators $T_h$
for different values of $\gamma\in(-\infty,+\infty)$ we notice that
the operator $T_h$ with the largest angle of sectorialilty occurs
when
\begin{equation}\label{e8-221}
    \gamma=-\frac{b}{2},
\end{equation}
and is found according to the formula
\begin{equation}\label{e8-222}
    \alpha=\arctan\frac{b}{1-b^2/4}.
\end{equation}
This follows from the formula \eqref{e8-201}, the fact that
$\gamma^2+\gamma\,b+1>0$ for all $\gamma$, and the formula
$$
\gamma^2+\gamma\,b+1=\left(\gamma+\frac{b}{2}\right)^2+\left(1-\frac{b^2}{4}\right).
$$
Now we will focus on the description of the parameter $h$ in the
restored operator $T_h$.

It was shown in \cite{ArTs0} that the quasi-kernel $\hat A$ of the
realizing system $\Theta$ from theorem \ref{t8-33} takes a form
\begin{equation}\label{e8-223}
\left\{\begin{array}{l}
  \widehat A y=-y^{\prime\prime}+qy \\
  y^{\prime}(a)=\eta y(a) \\
  \end{array} \right.,\;  \eta=\frac{\mu \RE h-|h|^2}{\mu-\RE h}
\end{equation}
On the other hand, since  $\sigma(t)$ is also the distribution
function of the positive self-adjoint operator,  we can conclude
that $\hat A$ equals to the operator $\ti B_\theta$ of the form
\eqref{e8-193}. This connection allows us to obtain
\begin{equation}\label{e8-224}
    \theta=\eta=\frac{\mu \RE h-|h|^2}{\mu-\RE h}.
\end{equation}
Assuming that
$$
h=x+iy
$$
we will use \eqref{e8-224} to derive the formulas for $x$ and $y$ in
terms of  $\gamma$. First, to eliminate parameter $\mu$, we notice
that \eqref{150} and \eqref{e5-62} imply
$$
W_\Theta(\lambda)= \frac{\mu -h}{\mu - \overline h}\,\,
\frac{m_\infty(\lambda)+ \overline
h}{m_\infty(\lambda)+h}=\frac{1-iV(z)}{1+iV(z)}.
$$
Passing to the limit when $z\to-\infty$ and taking into account that
$V(-\infty)=\gamma$ we obtain
$$
\frac{\mu -h}{\mu - \overline h}=\frac{1-i\gamma}{1+i\gamma}.
$$
Let us denote
\begin{equation}\label{e8-225}
   a=\frac{1-i\gamma}{1+i\gamma}.
\end{equation}
Solving \eqref{e8-225} for $\mu$ yields
$$
\mu=\frac{h-a\bar h}{1-a}.
$$
Substituting this value into \eqref{e8-224} after simplification
produces
$$
\frac{x+iy-a(x-iy)x-(x^2+y^2)(1-a)}{x+iy-a(x-iy)-x(1-a)}=\theta.
$$
After straightforward calculations targeting to represent numerator
and denominator of the last equation in standard form one obtains
the following relation
\begin{equation}\label{e8-226}
    x-\gamma\, y=\theta.
\end{equation}
It was shown in \cite{T87} that the $\alpha$-sectorialilty of the
operator $T_h$ and \eqref{e8-204} lead to
\begin{equation}\label{e8-227}
    \tan\alpha=\frac{\IM h}{\RE
    h+m_\infty(-0)}=\frac{y}{x+m_\infty(-0)}.
\end{equation}
Combining \eqref{e8-226} and \eqref{e8-227} one obtains
\begin{equation*}
    x-\gamma(x\tan\alpha+m_\infty(-0)\,\tan\alpha)=\theta,
\end{equation*}
or
$$
x=\frac{\theta+\gamma
m_\infty(-0)\,\tan\alpha}{1-\gamma\,\tan\alpha}.
$$
But $\tan\alpha$ is also determined by \eqref{e8-201}. Direct
substitution of
$$
\tan\alpha=\frac{b}{1+\gamma(\gamma+b)}
$$
into the above equation yields
$$
x=\theta+\frac{\big[\theta+m_\infty(-0)\big]b\gamma}{1+\gamma^2}.
$$
Using the short notation and finalizing calculations we get
\begin{equation}\label{e8-231}
   h=x+iy,\quad x=\theta+\frac{\gamma[\theta+m]b}{1+\gamma^2},\quad
y=\frac{[\theta+m]b}{1+\gamma^2}.
\end{equation}
At this point we can use \eqref{e8-231} to provide analytical and
graphical interpretation of the parameter $h$ in the restored
operator $T_h$. Let $$c=(\theta+m)b.$$ Again we consider three
subcases.
\begin{description}
  \item[Subcase 1] $b>2$ Using basic algebra we transform
\eqref{e8-231} into
\begin{equation}\label{e8-232}
    (x-\theta)^2+\left(y-\frac{c}{2}\right)^2=\frac{c^2}{4}.
\end{equation}
Since in this case the parameter $\gamma$ belongs to the interval in
\eqref{gamma-int},
\begin{figure}
  \begin{center}
  \includegraphics[width=80mm]{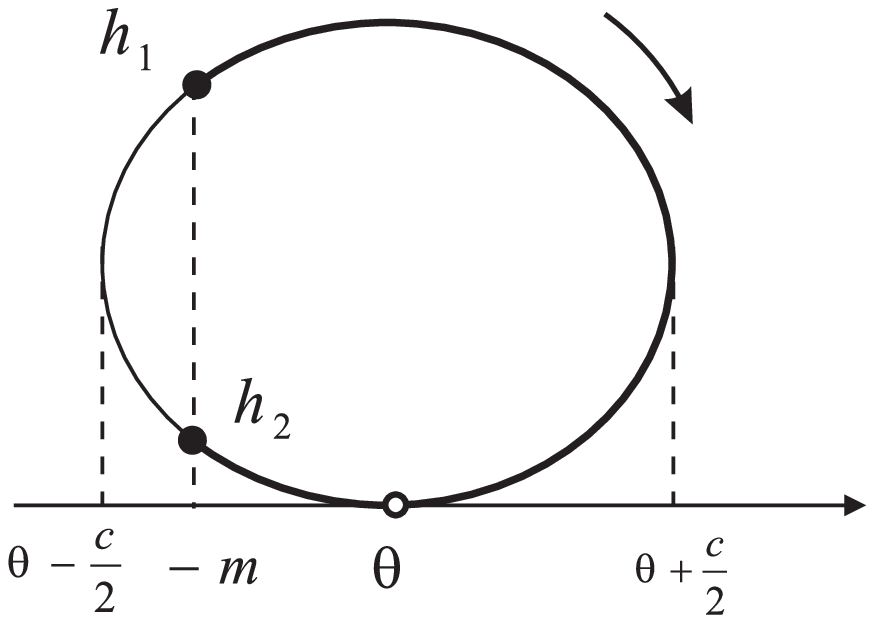}\\
  \caption{$b>2$}\label{fig8-1}
  \end{center}
\end{figure}
\begin{figure}
\begin{center}
\includegraphics[width=80mm]{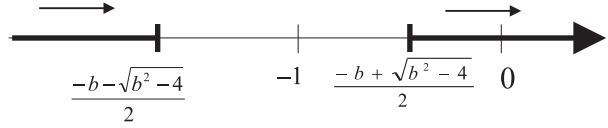}
\caption{$\gamma$ interval} \label{fig8-2}
\end{center}
\end{figure}
 we can see that $h$ traces the highlighted part
of the circle on the figure \ref{fig8-1} as $\gamma$ moves from
$-\infty$ towards $+\infty$. We also notice that the removed point
$(\theta,0)$ corresponds to the value of $\gamma=\pm\infty$ while
the points $h_1$ and $h_2$ correspond to the values
$\gamma_1=\frac{-b-\sqrt{b^2-4}}{2}$ and
$\gamma_2=\frac{-b+\sqrt{b^2-4}}{2}$, respectively (see figure
\ref{fig8-2}).

\begin{figure}
\begin{center}
\includegraphics[width=80mm]{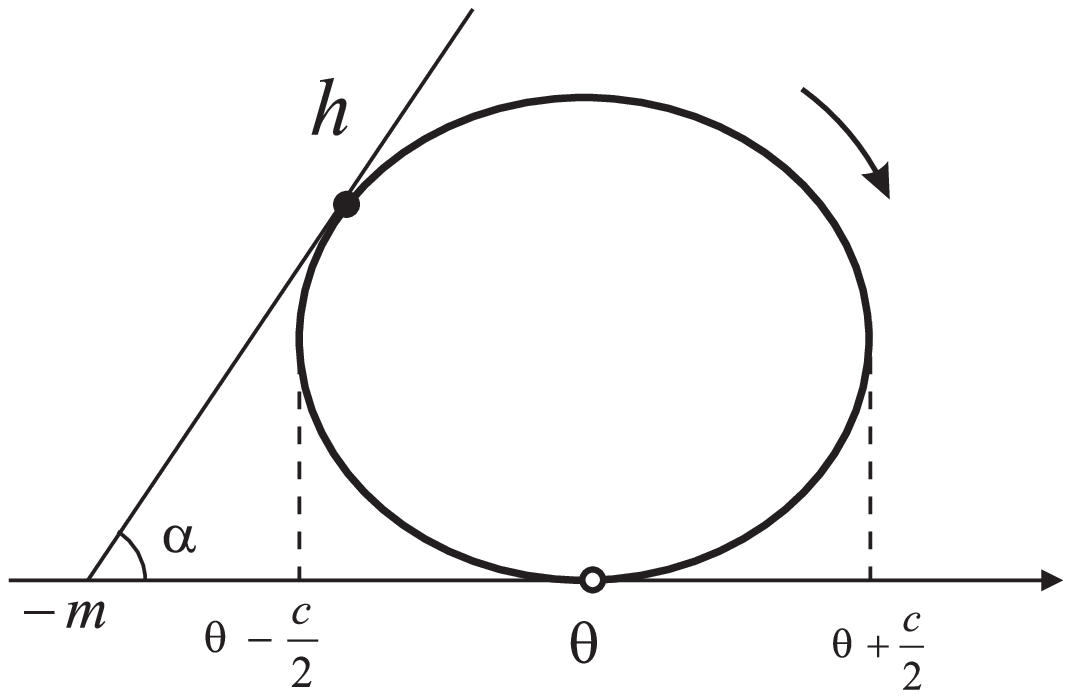}
\caption{$b<2$} \label{fig8-3}
\end{center}
\end{figure}

\item[Subcase 2] $b<2$ For every $\gamma\in(-\infty,+\infty)$ the
restored operator $T_h$ will be accretive and $\alpha$-sectorial for
some $\alpha\in(0,\pi/2)$. As we have mentioned above, the operator
$T_h$ achieves  the largest angle of sectorialilty  when
$\gamma=-\frac{b}{2}$. In this particular case \eqref{e8-231}
becomes
\begin{equation}\label{e8-233}
   h=x+iy,\quad x=\frac{\theta(4-b^2)-2b^2m}{4+b^2},\quad
y=\frac{4(\theta+m)b}{4+b^2}.
\end{equation}
The value of $h$ from \eqref{e8-233} is marked on the figure
\ref{fig8-3}.

\begin{figure}
\begin{center}
\includegraphics[width=80mm]{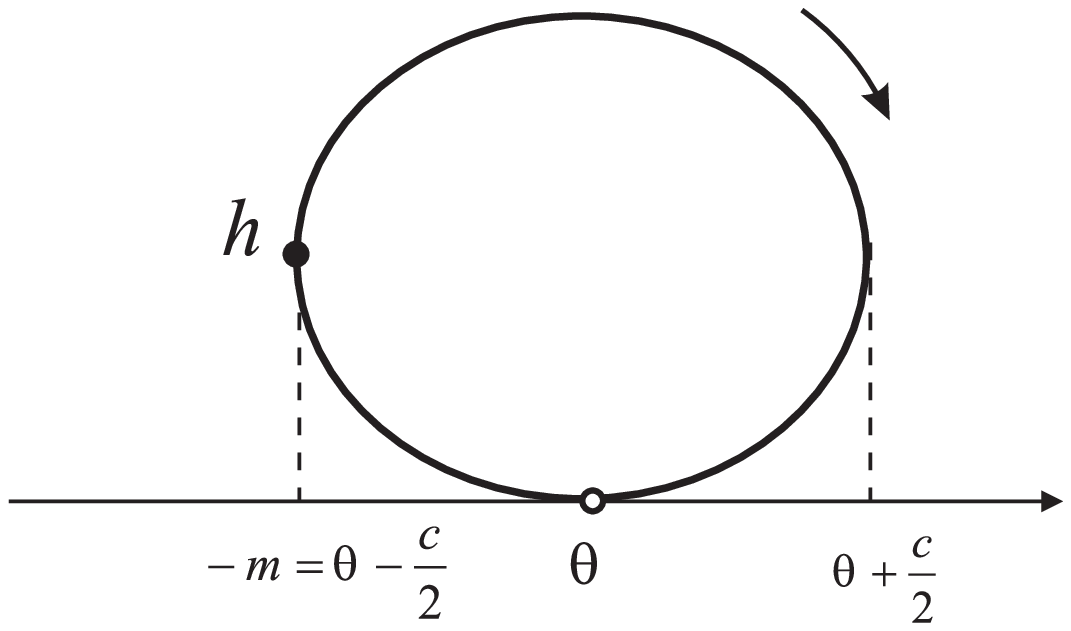}
\caption{$b=2$} \label{fig8-4}
\end{center}
\end{figure}

  \item[Subcase 3] $b=2$ The behavior of parameter $h$ in this case is depicted on the figure \ref{fig8-4}.
  It shows that in this case the function $V(z)$ can be realized using an extremal accretive
  $T_h$ when $\gamma=-1$. The value of the parameter $h$ according to \eqref{e8-231} then becomes
\begin{equation}\label{e8-234}
   h=x+iy,\quad x=-m,\quad y=\theta+m.
\end{equation}
Clockwise direction of the circle again corresponds to the change of
$\gamma$ from $-\infty$ to $+\infty$ and the marked value of $h$
occurs when $\gamma=-1$.
\end{description}

Now we consider the second case.

\textbf{Case 2}. Here we assume that $\int_0^{\infty }
\frac{d\sigma(t)}{t}=\infty$. This means that our function $V(z)$
belongs to the class $SL_0^K(R)$ and $b=\infty$. According to
theorem \ref{t8-34} and formulas \eqref{e8-200} and \eqref{e8-201},
the restored operator $T_h$ is accretive if and only if
$$\gamma\ge0,$$
and $\alpha$-sectorial if and only if $\gamma>0$. It directly
follows from \eqref{e8-201} that the exact value of the angle
$\alpha$ is then found from
\begin{equation}\label{e8-235}
    \tan\alpha=\frac{1}{\gamma}.
\end{equation}
The latter implies that the restored operator $T_h$ is extremal if
$\gamma=0$. This means that a function $V(z)\in SL_0^K(R)$ is
realized by a system with an extremal operator $T_h$ if and only if
\begin{equation}\label{e8-236}
    V(z)=\int\limits_0^\infty\frac{d\sigma(t)}{t-z}.
\end{equation}
On the other hand since $\gamma\ge0$ the function $V(z)$ is a
Stieltjes function of the class $S_{0}(R)$. Applying realization
theorems from \cite{DoTs} we conclude that $V(z)$ admits realization
by an accretive system $\Theta$ of the form (\ref{e6-3-2}) with
$\bA_R$ containing the  Krein-von Neumann extension $A_K$ as a
quasi-kernel. Here $A_K$ is defined by \eqref{139}. This yields
\begin{equation}\label{e8-237}
\theta=-m_\infty(-0)=-m.
\end{equation}
As in the beginning of the previous case we derive the formulas for
$x$ and $y$, where $h=x+iy$. Using \eqref{e8-224} and \eqref{e8-226}
leads to
\begin{equation}\label{e8-238}
\begin{cases}
    \theta=\frac{\mu x-(x^2+y^2)}{\mu-x},&\\
    x=\theta+\gamma y.&
\end{cases}
\end{equation}
Solving this system for $x$ and $y$ leads to
\begin{equation}\label{e8-239}
    x=\frac{\theta+\mu\gamma^2}{1+\gamma^2},\quad
y=\frac{(\mu-\theta)\gamma}{1+\gamma^2}.
\end{equation}
Combining \eqref{e8-238} and \eqref{e8-239} gives
\begin{equation}\label{e8-240}
    x=\frac{-m+\mu\gamma^2}{1+\gamma^2},\quad
y=\frac{(m+\mu)\gamma}{1+\gamma^2}.
\end{equation}
To proceed, we first notice that our function $V(z)$ satisfies the
conditions of theorem 4.8 of \cite{ArTs0}. Indeed, the inequality
\[
\mu \geq \frac{(\IM h)^2}{m_\infty(-0)+\RE h}+\RE h,
\]
turns into
$$
\mu=\frac{y^2}{x-m}+x,
$$
if you use $\theta=-m$ and the first equation in \eqref{e8-238}.
Applying theorem 4.8 of \cite{ArTs0} yields
\begin{equation}\label{e8-241}
\int\limits_0^\infty \frac{d\sigma(t)}{1+t^2}=\frac{\IM
h}{|\mu-h|^2} \left( \sup\limits_{y\in D(A_K)}\frac{|\mu y(a)-
y^{\prime}(a)|}{\left(\int\limits_a^\infty\left( |y(x)|^2+
|l(y)|^2\right)dx\right)^{\frac{1}{2}}}\right)^2.
\end{equation}
Taking into account that
$$
\mu y(a)- y^{\prime}(a)=(\mu+m) y(a)
$$
and  setting
\begin{equation}\label{e8-242}
    c^{1/2}=\sup\limits_{y \in D(A_K)}
\frac{|y(a)|}{\left(\int\limits_a^\infty\left(|y(x)|^2+
|l(y)|^2\right)dx\right)^{\frac{1}{2}}},
\end{equation}
we obtain
\begin{equation}\label{e8-243}
    \frac{\IM
h}{|\mu-h|^2} (\mu+m)^2\, c=\int\limits_0^\infty
\frac{d\sigma(t)}{1+t^2}.
\end{equation}
\begin{figure}
  \begin{center}
  \includegraphics[width=80mm]{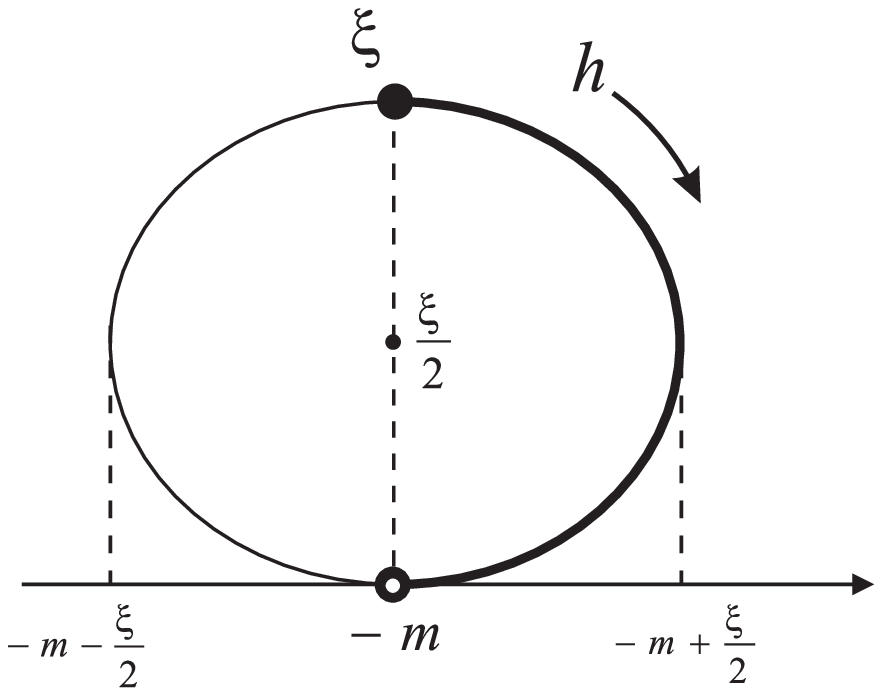}\\
  \caption{$b=\infty$}\label{fig8-5}
  \end{center}
\end{figure}
\begin{figure}
\begin{center}
\includegraphics[width=80mm]{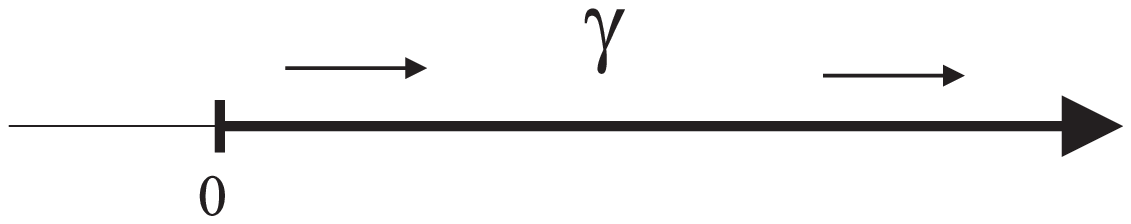}
\caption{$\gamma\ge0$} \label{fig8-6}
\end{center}
\end{figure}
Considering that $\IM h=y$ and combining \eqref{e8-243} with
\eqref{e8-240} we use straightforward calculations to get
$$
\mu=-m+\left(\frac{1}{\gamma\,c}\right)\,\int\limits_0^\infty
\frac{d\sigma(t)}{1+t^2}.
$$
Let
\begin{equation}\label{e8-244}
    \xi=\frac{1}{c}\int_0^\infty\frac{d\sigma(t)}{1+t^2}.
\end{equation}
Then the last equation becomes
\begin{equation}\label{e8-245}
\mu=-m+\frac{\xi}{\gamma}.
\end{equation}
Applying \eqref{e8-245} on \eqref{e8-240} yields
\begin{equation}\label{e8-246}
x=-m+\frac{\gamma\,\xi}{1+\gamma^2},\quad y=\frac{\xi}{1+\gamma^2},
\qquad \gamma\ge0.
\end{equation}
Following the previous case approach we transform \eqref{e8-246}
into
\begin{equation}\label{e8-247}
    (x+m)^2+\left(y-\frac{\xi}{2}\right)^2=\frac{\xi^2}{4}.
\end{equation}
The connection between the parameters $\gamma$ and $h$ in the
accretive restored operator $T_h$ is depicted in figures
\ref{fig8-5} and \ref{fig8-6}. As we can see $h$ traces the
highlighted part of the circle clockwise on the figure \ref{fig8-5}
as $\gamma$ moves from $0$ towards $+\infty$.

As we mentioned earlier  the restored operator $T_h$ is extremal if
$\gamma=0$. In this case formulas \eqref{e8-246} become
\begin{equation}\label{e8-247'}
x=-m,\quad y=\xi, \qquad \gamma=0,
\end{equation}
where $\xi$ is defined by \eqref{e8-244}.

\section{Realizing systems with Schr\"odinger operator}\label{s-6}

Now once we described all the possible outcomes for the restored
accretive operator $T_h$, we can concentrate on the main operator
$\bA$ of the system \eqref{e8-199}. We recall that $\bA$ is defined
by formulas \eqref{137} and beside the parameter $h$ above contains
also parameter $\mu$. We will obtain the behavior  of $\mu$ in terms
of the components of our function $V(z)$ the same way we treated the
parameter $h$. As before we consider two major cases dividing them
into subcases when necessary.

\textbf{Case 1}. Assume that $b=\int_0^\infty\frac{d\sigma(t)}{t}
<\infty$. In this case our function $V(z)$ belongs to the class
$SL^K_{01}(R)$. First we will obtain the representation of $\mu$ in
terms of $x$ and $y$, where $h=x+iy$. We recall that
$$
\mu=\frac{h-a\bar h}{1-a},
$$
where $a$ is defined by \eqref{e8-225}. By direct computations we
derive that
$$
a=\frac{1-\gamma^2}{1+\gamma^2}-\frac{2\gamma}{1+\gamma^2}i,\quad
1-a=\frac{2\gamma^2}{1+\gamma^2}+\frac{2\gamma}{1+\gamma^2}i,
$$
and
$$
h-a\bar
h=\left(\frac{2\gamma^2}{1+\gamma^2}x+\frac{2\gamma}{1+\gamma^2}y\right)+
\left(\frac{2}{1+\gamma^2}y+\frac{2\gamma}{1+\gamma^2}x \right)i.
$$
Plugging the last two equations into the formula for $\mu$ above and
simplifying we obtain
\begin{equation}\label{e8-248}
    \mu=x+\frac{1}{\gamma}\,y.
\end{equation}
We recall that during the present case $x$ and $y$ parts of $h$ are
described by the formulas \eqref{e8-231}.

Once again we elaborate in three subcases.
\begin{description}
  \item[Subcase 1] $b>2$ As we have shown this above, the formulas \eqref{e8-231}
can be transformed into equation of the circle \eqref{e8-232}. In
this case the parameter $\gamma$ belongs to the interval in
\eqref{gamma-int}, the accretive operator $T_h$ corresponds to the
values of $h$ shown in the bold part of the circle on the figure
\ref{fig8-1} as $\gamma$ moves from $-\infty$ towards $+\infty$.

Substituting the expressions for $x$ and $y$ from \eqref{e8-231}
into \eqref{e8-248} and simplifying we get
\begin{equation}\label{e8-250}
    \mu=\theta+\frac{(\theta+m)b}{\gamma}.
\end{equation}
The connection between values of $\gamma$ and $\mu$ is depicted on
the figure \ref{fig8-7}.
\begin{figure}
\begin{center}
\includegraphics[width=100mm]{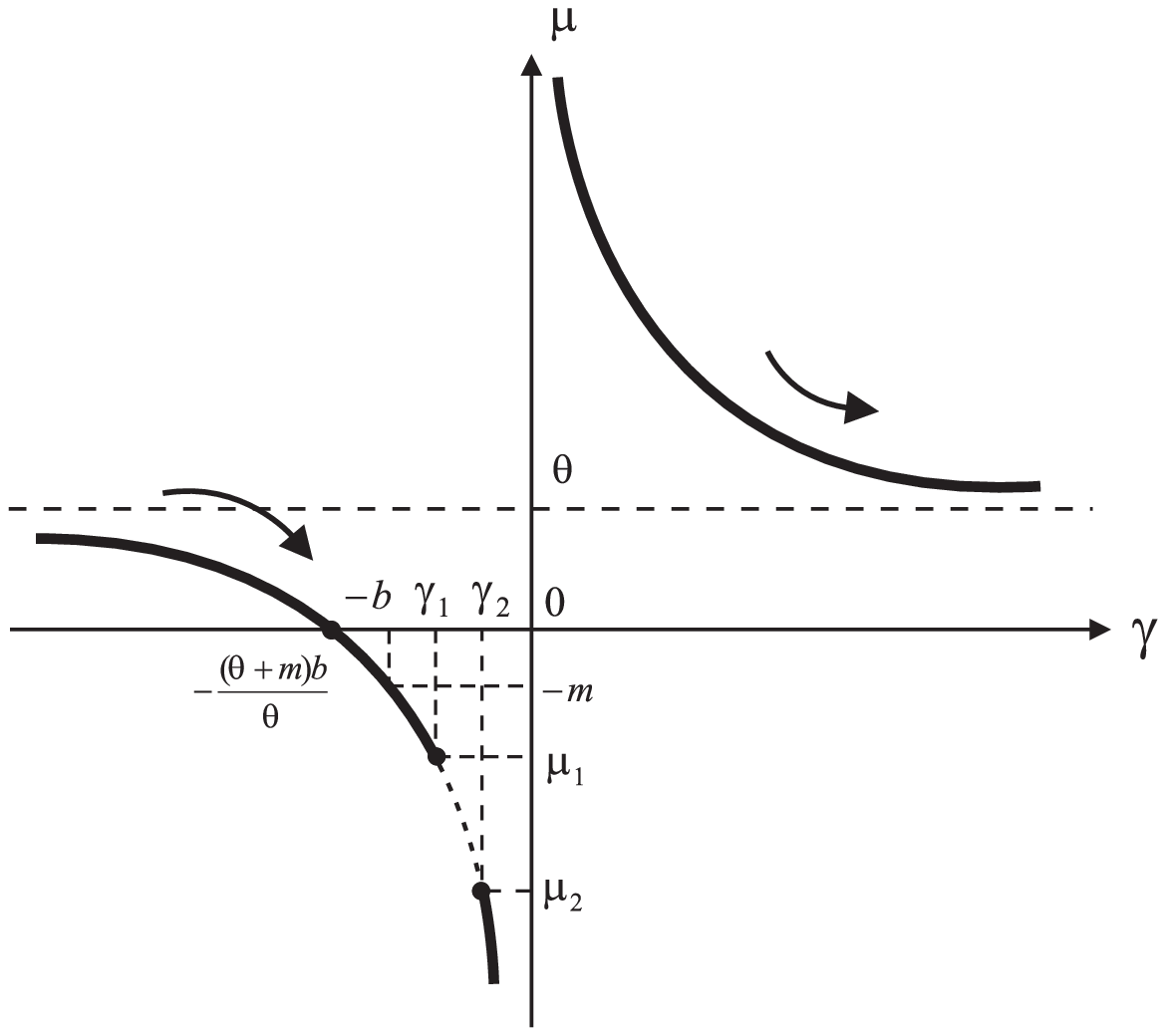}
\caption{$b>2$} \label{fig8-7}
\end{center}
\end{figure}
We note that $\mu=0$ when $\gamma=-\frac{(\theta+m)b}{\theta}$.
Also, the endpoints
$$
\gamma_1=\frac{-b-\sqrt{b^2-4}}{2}\quad\textrm{ and }\quad
\gamma_2=\frac{-b+\sqrt{b^2-4}}{2}
$$
of $\gamma$-interval \eqref{gamma-int} are responsible for the
$\mu$-values
$$
\mu_1=\theta+\frac{(\theta+m)b}{\gamma_1}\quad\textrm{ and }\quad
\mu_2=\theta+\frac{(\theta+m)b}{\gamma_2}.
$$
The values of $\mu$ that are acceptable parameters of operator $\bA$
of the restored system make the bold part of the hyperbola on the
figure \ref{fig8-7}. It follows from theorem \ref{t8-20} that the
operator $\bA$ of the form \eqref{137} is accretive if and only if
$\gamma\ge0$ and thus $\mu$ sweeps the right branch on the
hyperbola. We note that figure \ref{fig8-7} shows the case when
$-m<0$, $\theta>0$, and $\theta>-m$. Other possible cases, such as
($-m<0$, $\theta<0$, $\theta>-m$), ($-m<0$, $\theta=0$), and ($m=0$,
$\theta>0$) require corresponding adjustments to the graph shown in
the picture \ref{fig8-7}.

\item[Subcase 2] $b<2$ For every $\gamma\in(-\infty,+\infty)$ the
restored operator $T_h$ will be accretive and $\alpha$-sectorial for
some $\alpha\in(0,\pi/2)$. As we have mentioned above, the operator
$T_h$ achieves  the largest angle of sectorialilty  when
$\gamma=-\frac{b}{2}$. In this particular case \eqref{e8-231}
becomes
\begin{equation*}
    h=x+iy,\quad    x=\frac{\theta(4-b^2)-2b^2m}{4+b^2},\quad
y=\frac{4(\theta+m)b}{4+b^2}.
\end{equation*}
Substituting $\gamma=b/2$  into \eqref{e8-248}  we obtain
\begin{equation}\label{e8-251}
    \mu=-(\theta+2m).
\end{equation}
This value of $\mu$ from \eqref{e8-251} is marked on the figure
\ref{fig8-8}. The corresponding operator $\bA$ of the realizing
system is based on these values of parameters $h$ and $\mu$.
\begin{figure}
\begin{center}
\includegraphics[width=100mm]{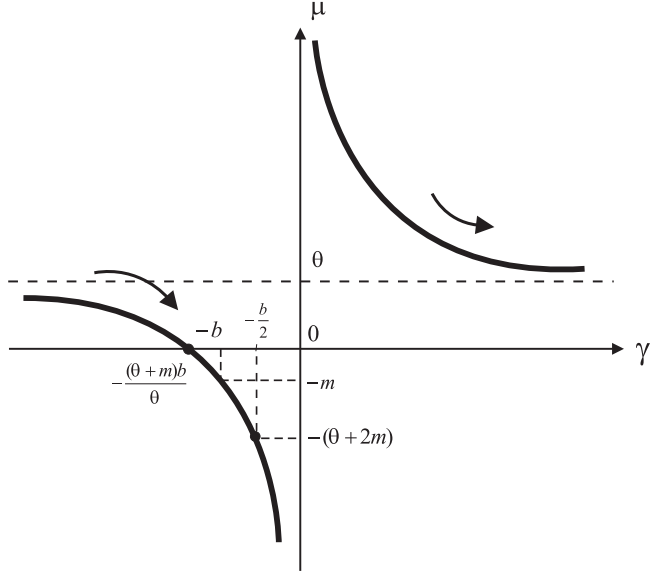}
\caption{$b<2$ and $b=2$} \label{fig8-8}
\end{center}
\end{figure}

  \item[Subcase 3] $b=2$ The behavior of parameter $\mu$ in this case is also shown on the figure \ref{fig8-8}.
  It was shown above that in this case the function $V(z)$ can be realized using an extremal accretive
  $T_h$ when $\gamma=-1$. The values of the parameters $h$ and $\mu$  then become
\begin{equation*}
   h=x+iy,\quad x=-m,\quad y=\theta+m,\quad \mu=-(\theta+2m).
\end{equation*}
The value of $\mu$ above is marked on the left branch of the
hyperbola and occurs when $\gamma=-1=-b/2$.
\end{description}

\textbf{Case 2}. Again we assume that $\int_0^{\infty}
\frac{d\sigma(t)}{t}=\infty$. Hence $V(z)\in SL_0^K(R)$ and
$b=\infty$. As we mentioned above the restored operator $T_h$ is
accretive if and only if $\gamma\ge0$ and $\alpha$-sectorial if and
only if $\gamma>0$. It is extremal if $\gamma=0$. The values of $x$,
$y$, and $\mu$ were already calculated and are given in
\eqref{e8-246} and \eqref{e8-245}, respectively. That is
$$
x=-m+\frac{\gamma\,\xi}{1+\gamma^2},\quad
y=\frac{\xi}{1+\gamma^2},\quad  \mu=-m+\frac{\xi}{\gamma},\qquad
\gamma\ge0.
$$
where $\xi$ is defined in \eqref{e8-244}. Figure \ref{fig8-9} gives
graphical representation of this case. Only the right bold branch of
hyperbola shows the values of $\mu$ in the case $b=\infty$. If $m=0$
then
$$
\mu=\frac{\xi}{\gamma}
$$
and the graph should be adjusted accordingly.
\begin{figure}
\begin{center}
\includegraphics[width=100mm]{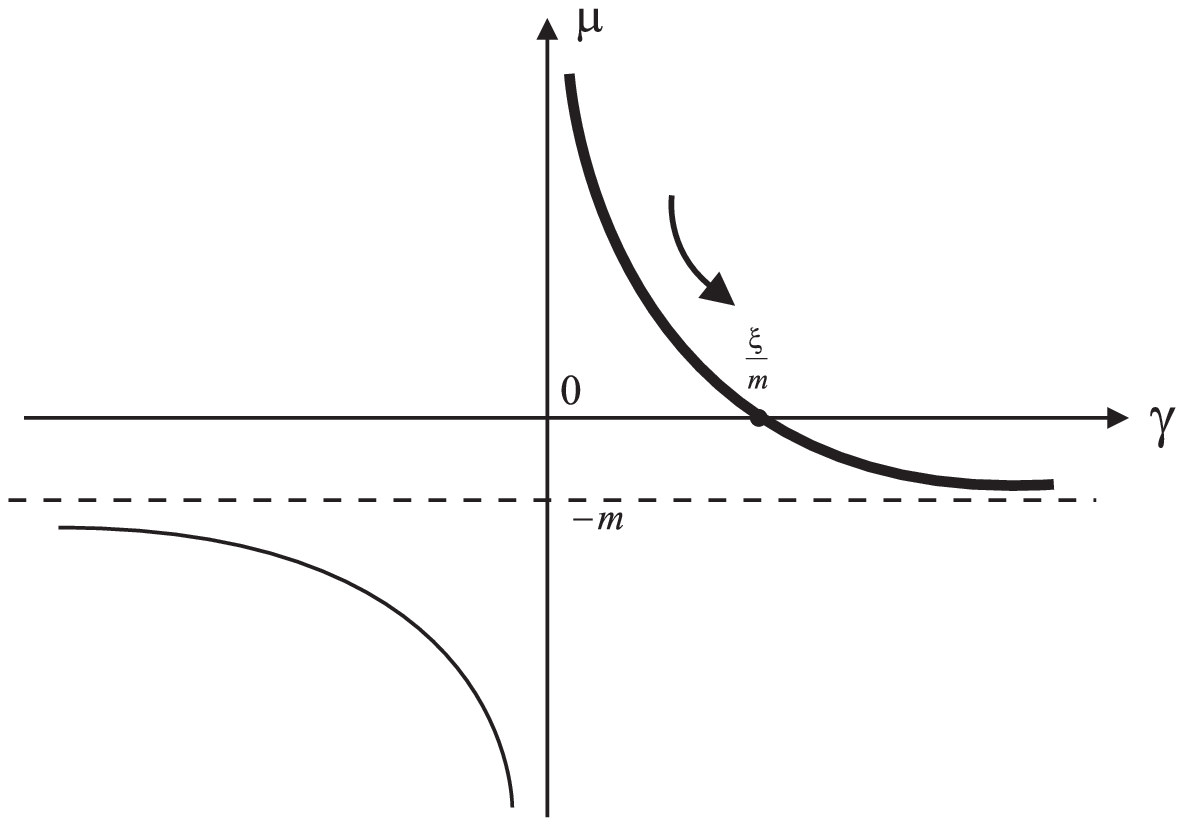}
\caption{$b=\infty$} \label{fig8-9}
\end{center}
\end{figure}

In the case when $\gamma=0$ and $T_h$ is extremal we have
\begin{equation}\label{e8-252}
    x=-m,\quad y=\xi,\quad \mu=\infty,\quad h=-m+i\xi,
\end{equation}
and according to \eqref{137} we have
\begin{equation}\label{e8-253}
    \bA y=-y''+q(x)y+[(-m+i\xi)y(a)-y'(a)]\delta(x-a),
\end{equation}
that is the main operator of the realizing system.

\subsection*{Example}
We conclude this paper with simple illustration. Consider  a
function
\begin{equation}\label{e-ex1}
    V(z)=\frac{i}{\sqrt z}.
\end{equation}
A direct check confirms that $V(z)$ is a Stieltjes function. It was
shown in \cite{Na68} (see pp. 140-142) that the inversion formula
\begin{equation}\label{e-ex2}
    \sigma(\lambda)=C+\lim_{y\to0}\frac{1}{\pi}\int_0^\lambda\IM\left(\frac{i}{\sqrt{}x+iy}\right)\,dx
\end{equation}
describes the distribution function for a self-adjoint operator
$$
\left\{ \begin{array}{l}
\ti B_0y=-y^{\prime\prime} \\
 y'(0)=0. \\
 \end{array} \right.
$$
The corresponding to $\ti B_0$ symmetric operator is
\begin{equation}\label{tiB0}
\left\{ \begin{array}{l}
  B_0y=-y^{\prime\prime} \\
 y(0)=y'(0)=0. \\
 \end{array} \right.
\end{equation}
 It was also shown
in \cite{Na68} that $\sigma(\lambda)=0$ for $\lambda\le0$ and
\begin{equation}\label{e-ex3}
    \sigma'(\lambda)=\frac{1}{\pi\sqrt\lambda}\textrm{\ \  for\ \
    }\lambda>0.
\end{equation}
By direct calculations one can confirm that
$$
V(z)=\int_0^\infty\frac{d\sigma(t)}{t-z}=\frac{i}{\sqrt z},
$$
and that
$$
\int_0^\infty\frac{d\sigma(t)}{t}=\int_0^\infty\frac{dt}{\pi
t^{3/2}}=\infty.
$$
It is also clear that the constant term in the integral
representation \eqref{e8-94} is zero, i.e. $\gamma=0$.

Let us assume that $\sigma(t)$ satisfies our definition of spectral
distribution function of the pair $B_0$, $\ti B_0$ given in the
section \ref{s-5}. Operating under this assumption, we proceed to
restore parameters $h$ and $\mu$ and apply formulas \eqref{e8-246}
for the values $\gamma=0$ and $m=-\theta=0$. This yields $x=0$. To
obtain $y$ we first find the value of
$$
\int_0^\infty\frac{d\sigma(t)}{1+t^2}=\frac{1}{\sqrt2},
$$
and then use formula \eqref{e8-242} to get the value of $c$. This
yields $c=1/\sqrt2$. Consequently,
$$
\xi=\frac{1}{c}\int_0^\infty\frac{d\sigma(t)}{1+t^2}=1,
$$
and hence $h=yi=i$. From \eqref{e8-245} we have that $\mu=\infty$
and \eqref{e8-253} becomes
\begin{equation}\label{e-ex4}
    \bA\, y=-y''+[iy(0)-y'(0)]\delta(x).
\end{equation}
The operator $T_h$ in this case is
$$
\left\{ \begin{array}{l}
 T_h y=-y^{\prime\prime} \\
 y'(0)=iy(0). \\
 \end{array} \right.
$$
 The channel vector $g$ of the
form \eqref{146} then equals
\begin{equation}\label{e-ex5}
    g=\delta(x),
\end{equation}
satisfying
$$
\IM\bA=\frac{\bA - \bA^*}{2i}=KK^*=(.,g)g,
$$
 and channel operator $Kc=cg$, ($c\in\dC$) with
\begin{equation}\label{e-ex6}
    K^*y=(y,g)=y(0).
\end{equation}
The real part of $\bA$
$$
\RE\bA\, y=-y''-y'(0)\delta(x)
$$
contains the self-adjoint quasi-kernel
$$
\left\{ \begin{array}{l}
 \widehat A y=-y^{\prime\prime} \\
 y'(0)=0. \\
 \end{array} \right.
$$
 A system of the Liv\u sic type  with Schr\"odinger
operator of the form \eqref{e8-199} that realizes $V(z)$ can now be
written as
$$
\Theta= \begin{pmatrix} \bA&K&1\cr \calH_+ \subset L_2[a,+\infty)
\subset \calH_-& &\dC\cr \end{pmatrix}.
$$
where $\bA$ and $K$ are defined above. Now we can back up our
assumption on $\sigma(t)$ to be the spectral distribution function
of the pair $B_0$, $\ti B_0$. Indeed, calculating the function
$V_\Theta(z)$ for the system $\Theta$ above directly via formula
\eqref{1501} with $\mu=\infty$ and comparing the result to $V(z)$
gives the exact value of $h=i$. Using the reasoning  of remark
\ref{r5-5} we confirm that $\sigma(t)$ is the spectral distribution
function of the pair $B_0$, $\ti B_0$.

\begin{remark}
All the derivations above can be repeated for a Stieltjes like
function
$$V(z)=\gamma+\frac{i}{\sqrt z},\qquad -\infty<\gamma<+\infty,\;\gamma\ne0$$
with very minor changes. In this case the restored values for $h$
and $\mu$ are described as follows:
$$
h=x+iy,\quad x=\frac{\gamma}{1+\gamma^2},\quad
y=\frac{1}{1+\gamma^2},\quad \mu=\frac{1}{\gamma}.
$$
The dynamics of changing $h$ according to changing $\gamma$ is
depicted on the figure \ref{fig8-5} where the circle has the center
at the point $i/2$ and radius of $1/2$. The behavior of $\mu$ is
described by a hyperbola $\mu=1/\gamma$ (see figure \ref{fig8-9}
with $\theta=0$). In the case when $\gamma>0$ our function becomes
Stieltjes and the restored system $\Theta$ is accretive. The
operators $\bA$ and $K$ of the restored system are given according
to the formulas \eqref{137} and \eqref{148}, respectively.
\end{remark}


\end{document}